\documentclass{amsart}
\usepackage{url}
\usepackage{verbatim}
\usepackage{alltt}
\usepackage{graphicx}
\graphicspath{{./figure/}}
\usepackage{amssymb}
\usepackage{multirow}

\newtheorem{theorem}{Theorem}
\newtheorem{lemma}[theorem]{Lemma}
\theoremstyle{remark}
\newtheorem{remark}{Remark}[section]
\newtheorem{example}{Example}[section]
\theoremstyle{definition}
\newtheorem{definition}[theorem]{Definition}

\newcommand{\abs}[1]{\vert#1\vert}

\newcommand{\midB}{\mathrel{} \middle | \mathrel{} }

\newcommand{\grad}{\nabla}
\newcommand{\bq}{\begin{equation}}
\newcommand{\eq}{\end{equation}}
\newcommand{\R}{\mathbb{R}}

\newcommand{\Rd}{\R^d}

\newcommand{\dx}{h}
\newcommand{\Ind}{K}

\begin{document}

\title[Efficient Linear Programming for Optimal Transportation]{An efficient Linear Programming method for  Optimal Transportation}
\author{Adam~M. Oberman}
\author{Yuanlong Ruan}
\thanks{The first author thanks Mikhael Mamaev who contributed to this project as part of an NSERC Undergraduate Summer Research Award
 (USRA)}
\date{\today}
\begin{abstract}
An efficient  method  for computing solutions to the  Optimal Transportation (OT) problem with a wide class of cost functions is presented.
  The standard linear programming (LP)  discretization of the continuous problem becomes intractible for moderate grid sizes.   A grid refinement method results in a linear cost algorithm. 
 Weak convergence of solutions is stablished.  Barycentric projection of transference plans is used to improve the accuracy of solutions. The method is applied to more general problems, including partial optimal transportation, and barycenter problems.  Computational examples validate the accuracy and efficiency of the method.   Optimal maps between nonconvex domains, partial OT free boundaries, and high accuracy barycenters are presented. \end{abstract}

\maketitle
\tableofcontents

\section{Introduction}
The theory of Optimal Transportation (OT)  is a powerful tool which defines a distance on probability measures, called the Wasserstein distance.  It has profoundly impacted such fields as  fluid mechanics, nonlinear diffusions,  differential geometry, and it has  found applications in areas such as reflector design, economics, astrophysics, etc.\ \cite{EvansSurvey, Villani2003Topics,Villani2008Optimal}.  The related barycenter problem has been used for comparison of histograms of features for big data problems,  and shape interpolation for computer graphics  \cite{agueh2011barycenters, solomon2015convolutional, cuturi2015smoothed}.   

The Wasserstein distance cannot be evaluated analytically in most cases, and so numerical methods are needed.  Existing numerical methods are either inaccurate, very costly (polynomial time), or too restrictive on the costs and measures to be widely applicable.   These methods do not meet need the needs of modern applications, which involve large problem sizes or require the solution of large numbers of OT problems.

We introduce an efficient  Linear Programming (LP) method for solving Optimal Transportation problems.  Our method is not restricted to quadratic cost functions: it allows for a wide class of cost functions.    By efficient, we mean linear cost in terms of solution time and memory requirements (see Table~\ref{table:OTPerformance} below).  In practise we can compute solutions with problem sizes going up to half a million  variables on a laptop computer using academic or commerical optimization software.  Available parallel LP solvers allow the method to scale to even larger problem sizes.   

Two types of error are present in approximation problems which discretize a continous (infinite dimensional) optimization problem.  This first is the discretization error: the different between the exact solution of the discrete problem (at a given resolution) and the solulution of the limiting problem.  Our discretization of the infinite dimensional LP is natural, allowing a weak convergence proof to be established using available stability results.   
The second type of error relates to how precisely the finite dimensional optimization problem is solved.  In contrast to many existing methods, which solve surrogate or approximate optimzation problems, 
 the method presented here compute the optimal solution of the discrete linear transportation problem, to within standard solver tolerances.  

A limitation  of the  Linear Programming approach to Optimal Transportation, compared to the Partial Differential Equations methods discussed below, is that maps are approximated by plans, which are multivalued (see Figure~\ref{fig:Maps and Plans}).  
We overcome this limitation using a theoretical tool, barycentric projection, which recovers the approximate map and dramatically improves the accuracy of solutions.

The efficiency and accuracy of the method reveals solution features not otherwise  available, including optimal maps between nonconvex sets, and for non-quadratic cost functions.
The method is easily generalized to related problems, allowing for the computation of accurate free boundaries in partial optimal transportation, and high resolution  barycenters for shapes.

\subsection{Background on the Optimal Transportation problem}
The Monge formulation of the Optimal Transportation problem which seeks optimal \emph{maps} while the Kantorovich formulation  seeks optimal \emph{transference plans}.  Transference plans are a weaker notion of solution which can be computed using linear programming.

\begin{definition}[Givens for the OT problem]\label{defn:data}
Given two probability measures, 
\[
\text{ $\mu, \nu$  with bounded  supports   $X, Y \subset \Rd$, respectively }
\]
and the  cost function 
\[
c(x,y): X\times Y\to \R.
\]   
If the measures $\mu$ and $\nu$ are absolutely continuous with respect to Lebesgue measure, write 
\[
d\mu(x) = f(x) dx,
\qquad
d\nu(y) = g(y) dy
\]
for positive, Lebesgue measurable functions $f$ and $g$. 
\end{definition}

 The goal is to rearrange one measure into the other, while minimizing the cost of mapping $x$ to $y$, weighted by the amount of mass transported.  There are two notions of rearrangement.

\subsection{The Monge formulation}
In the Monge formulation, we consider the class of measurable, one-to-one mappings $T:\Rd\to\Rd$ which rearrange $\mu$ into $\nu$.  If the measures $\mu$ and $\nu$ have smooth densities, 
and if $T$ is continuously differentiable, the change of variable formula from multivariable calculus, leads to the 
rearrangement condition,
\bq\label{determinant}
f(x) = g(T(x)) \abs{ \det \grad T(x)}
\eq
which is a fully nonlinear constraint on the mapping $T$.
For general measures, the rearrangement condition is given by 
\bq\label{rearrange}
\nu[B] = \mu[T^{-1}(B)], 
\qquad
\text{ for any measureable set } B \subset Y.
\eq
We write $\nu = T\# \mu$, and say that $T$ transports $\mu$ onto $\nu$.   

Monge's problem is to minimize the total work corresponding to a map $T$ over the set of all measurable maps $T$ which transport $\mu$ onto $\nu$. 
\bq \label{MongeProblem}
\begin{aligned}
\text{Minimize   }   & I[T] = \int_X c(x,T(x)) d\mu(x)
\\  \text{ Subject to: } 
& T\# \mu = \nu
\end{aligned}
\eq

\subsection{The Kantorovich Formulation}\label{sec:Kant}
The Kantorovich formulation  recasts the Optimal Transportation problem as an infinite dimensional Linear Program.  Refer again to \cite{Villani2003Topics, EvansSurvey}.
  
A transference plan generalizes a mapping, in order to allow for mass from  a point $x$ to be split into multiple parts.   A transference plan is a probability measure, $\pi$, on the product space $X\times Y$, whose marginals are $\mu$ and $\nu$.   This means that
\[
\pi [ A\times Y] =\mu(A),
\quad
\pi( X\times B)  =\nu(B)
\]
for all measurable subsets $A$ of $X$ and $B$ of $Y$.
The set of transference plans is written
\bq\label{Tplans}
\Pi(\mu,\nu)  =
\left\{ 
\pi  \in \mathcal{P}(X\times Y) \midB
\text{ marginals of $\pi$ are } \mu, \nu
\right \}
\eq
Kantorovich's formulation of the optimal transportation problem take the form of an infinite dimensional linear program
\bq\label{KantLP}\tag{KLP}
\text{Minimize   }  I[\pi]  =\int_{X\times Y}c(x,y)
d\pi(x,y), \quad \text{ for } \pi \in \Pi(\mu,\nu) 
\eq
Since the constraints are linear, and the cost function is also linear, this is an infinite dimensional linear program.    Under broad conditions, this problem has a minimal value, which is called the \emph{optimal transportation cost} between $\mu$ and $\nu$.   However, in general, there may be more than one optimal transference plan.

\begin{remark}[When an optimal plan is a map]
For many costs, the transference plan solution from \eqref{KantLP} is given by a map.
This is the case for 
\[
c(x,y) =h(\abs{x-y}), 
\]
where $h$ is strictly convex and grows at super linear speed.   In particular, $c(x,y) =\abs{x-y}^{p}$ for $p>1$, which includes the important case $p=2$. 

More generally, the \emph{twist} condition, which requires that $g(y) = \grad_x c(x,y)$ must be an injective function of $y$, ensures that the optimal plan is a map, see \cite{gangbo1996geometry, carlier2003duality}, or \cite[Chapter 10]{Villani2008Optimal}.  
For other cost functions, an optimal  transference plan  need not be a map.  In particular, this is the case for the Monge cost  $c(x,y) = \abs{x-y}$.
\end{remark}

\begin{example}[Monge-Amp\`ere Partial Differential Equation]\label{ex:MongeAmpere}
When   $c(x,y) = \abs{x-y}^2$, the densities are smooth, and $Y$ is convex, 
the Monge problem becomes a fully nonlinear Monge-Amp\`ere PDE with unusual boundary conditions  
\cite[Section 4]{EvansSurvey} or \cite[Chapter 4]{Villani2003Topics}.  
The optimal  map $T = \grad u$ is given by the gradient of a convex function, so that \eqref{determinant} becomes
\[
\det( D^2 u(x)) = \frac{f(x)}{g(\grad u(x))}
\]
along with the conditions $\grad u(X)= Y$ and the constraint that $u(x)$ be convex.  
\end{example} 

\begin{example}[Assignment]\label{example:assignment}
A special case of the Monge problem is the \emph{assignment problem} from combinatorial optimization \cite[Ch 17 and Ch 21]{schrijver2003combinatorial}, which is to find a minimum cost bipartite  matching between two sets.  See Figure~\ref{fig:Maps and Plans}(a). 
If  $X$ and $Y$ are discrete with $n$ points each having the same mass, then 
\[
\mu = \frac{1}{n} \sum_{i=1}^n \delta_{x_i},
\qquad
\nu = \frac{1}{n} \sum_{i=1}^n \delta_{y_i},
\]
where $\delta_x$ represents the Dirac mass at the point $x$. 
In this case, the Monge problem \eqref{MongeProblem}, becomes the assignment problem 
\[
\text{Minimize }   \sum_{i=1}^n c(x_i,y_{T(i)}),
\quad  
\text{ over all permutations }  T \text{ of } \{1,\dots, n\}
\]
\end{example}

\begin{example}[The discrete transportation problem]
A special case of \eqref{KantLP} is the transportation Linear Program, see Figure~\ref{fig:Maps and Plans}(b). 
Given discrete probability measures 
\[
\mu =\sum_{i=1}^n \mu_i \delta_{x_i},
\quad
\nu =  \sum_{j=1}^m \nu_j \delta_{y_j},
\qquad
\mu_i, \nu_j \ge 0, 
~
\sum_{i=1}^n \mu_i  =  \sum_{j=1}^m \nu_j = 1.
\]
The cost function is given by  the non-negative $n\times m$ matrix, $c = (c_{ij})$.  
A discrete transference plan, $\pi$, is a non-negative $n\times m$ matrix whose marginals are $\mu$ and $\nu$.  The set of transference plans \eqref{Tplans} becomes 
\bq\label{FDproject}
\Pi(  \mu,\nu)  =
\left\{ 
\pi = (\pi_{ij})  \midB 
\sum_{j=1}^m \pi_{ij}  =\mu_i, 
\quad 
\sum_{i=1}^n \pi_{ij}  =\nu_j, 
\quad 
\pi_{ij} \ge 0
\right \}
\eq

The transportation linear program is given by 
\bq \label{LPfull} \tag{LP}
\text{Minimize   } I[\pi] =  \sum_{i=1}^n\sum_{j=1}^m c_{ij} \pi_{ij}, 
\quad
\text{ for  } \pi \in \Pi(\mu, \nu)
\eq
The problem \eqref{LPfull} is easily written in the standard form for a linear program.
\end{example} 

\begin{remark}\label{rem:quadratic}
Notice that the number of variables used to represent the densities in \eqref{LPfull} is $n + m$,
and the number of variables for the plan $\pi$ is $n m$.  So the size of the linear programming problem grows \emph{quadratically} in the number of variables. 
\end{remark}

\begin{example}[Transportation recovers Assignment]  Consider the discrete transportation problem \eqref{LPfull} where the measures 
are given by the equal weight measures from Example~\ref{example:assignment}.  In this case, (after removing the factor of $1/n$)  the set of  plans, $\Pi$, becomes the set of  bistochastic matrices. By Choquet's theorem, the solution of \eqref{LPfull} are extremal points of bistochastic matrices.  By Birkhoff's theorem, these  are given by permutation matrices.  Thus the optimal transference plans in Kantorovich's problem coincide with the permutation maps in Monge's problem, and the solutions are the same. 
\end{example} 

\begin{figure}[ptb]
\begin{minipage}{0.5\linewidth}
\includegraphics[width=\linewidth]{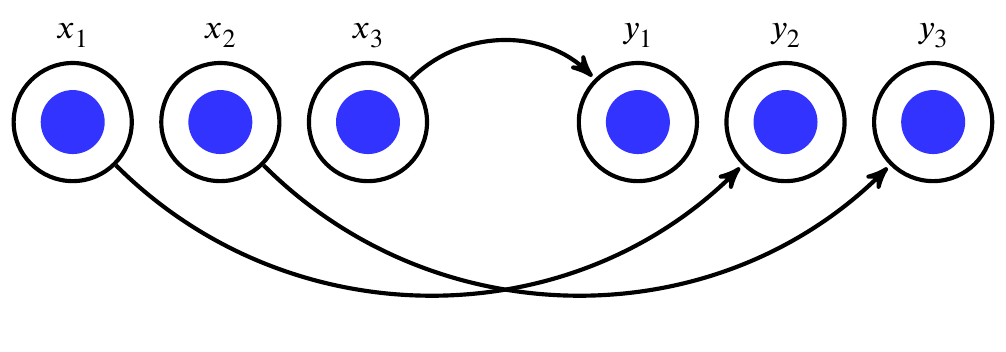}
\includegraphics[width=\linewidth]{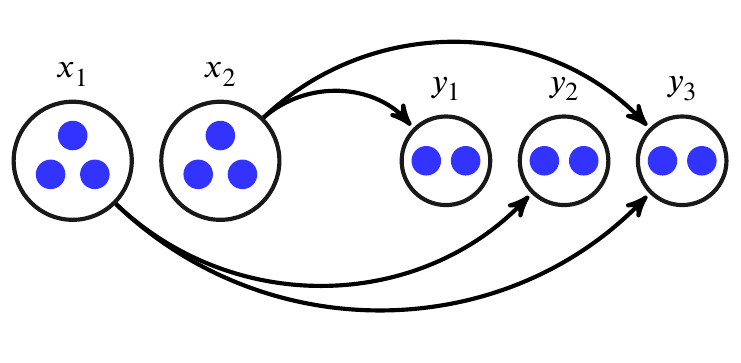}
\end{minipage}
\vspace{-.25in}
\caption{(a) Illustration of a mapping from three points to three points.  (b)  An optimal transference plan can split points.}%
\label{fig:Maps and Plans}
\end{figure}

\subsection{Related work} \label{sec:related work}
Rigorous approaches to solving the Optimal Transportation rely on different theoretical notions of the solution.   Many solution methods are specific to the case of quadratic cost (distance squared), or the Monge cost (distance).

\subsubsection*{Partial Differential Equations}
First consider the important special case of quadratic cost,  $c(x,y) = \abs{x-y}^2$.
The early Benamou-Brenier  formulation leads to a fluid mechanics solver, by adding a synthetic time variable to the problem~\cite{BenamouBrenier}, adding one dimemsion to the problem.   The Monge-Amp\`ere Partial Differential formulation was recently used to solve the OT problem using a convergent finite difference method \cite{Benamou2014107}.  This method  places  regularity requirements on the densities,  one of which must have convex support.   Our method has comparable accuracy to the PDE approach, see \S\ref{sec:accuracy}.

\subsubsection*{Earth Mover Distance}
For the linear cost case, $c(x,y) = \abs{x-y}$, fast methods for computing approximate solutions are available~\cite{andoni2008earth}.   These methods work by computing an approximate embedding of the 1-Wasserstein distance into Euclidean space with the $L^1$ metric.  However the embedding induces a distortion of the metric which limits the accuracy of the method~\cite{naor2007planar}.

\subsubsection*{Entropic regularization.} 
Entropic regularization methods are a recently introduced method which modify the the optimization problem by adding a small multiple of an entropy term~\cite{cuturi2013sinkhorn, cuturi2015smoothed} to the objective (or cost) function.  
 The regularized problem can be solved by a Bregman iterative solution method  \cite{benamou2015iterative}.    However, the number of iterations required for convergence increases as the regularization parameter goes to zero.   
The recent article \cite{solomon2015convolutional} applied the method to large scale geometric problems.
  We compare the performance of our method to the entropic regularization method in \S\ref{sec:performance}.   
  Our method also avoids the blurring of barycenters introduced by the entropic regularization, 
 see \S\ref{sec:barcyenterN}.

\subsubsection*{Linear Programming and Combinatorial Optimization} 
The linear programming problem implemented directly is simply too costly to be effective for solving the OT problem. For example, in the early paper   \cite{ruschendorf2000numerical}, problem sizes of up to  $N = 15$ were solved.  Current implementations still become very costly after $N = 7200$, see Table~\ref{table:OTPerformance}.  This is because the size of the problem grows quadratically in the number of variables, see Remark~\ref{rem:quadratic}.
 An alternative approach is to instead solve the assignment problem \cite{burkard1999linear}, 
 approximating fractional weights for the measures by multiple assignents, see for example see  \cite{bertsekas1989auction}.   However combinatorial optimization algorithms for the assigment problem do not seem to perform better than linear programming.  In best cases, the complexity is worse than quadratic \cite{schrijver2003combinatorial}.

\subsubsection*{Multiscale solvers} 
Multiscale solvers have previously  been used to improve solver performance, but without achieving linear efficiency. 
Merigot  solved a sequence of OT problems where the target is a sum of Diracs, improving performance by  an order of magnitude \cite{merigot2011multiscale}.   
Benamou and coauthors combines the approach of \cite{benamou2015iterative} with a grid refinement procedure to solve larger scale problems in  \cite{benamou2015numerical}. In \cite{carlier2014numerical} a one step grid refinement was used to find the support of the barcyenter.   Schmitzer \cite{schmitzer2015sparse, schmitzer2013hierarchical} used a grid refinement procedure which was be applied to both LP and combinatorial optmization solvers, 
improving performance by one order of magnitude.

\section{Discretization and convergence}\label{sec:discretization} 
In this section we perform the discretization  which reduces  \eqref{KantLP} to a finite dimensional Linear Program \eqref{LPfull}.   We prove that the solutions of the discrete problem converge weakly to the solution of \eqref{KantLP} as the grid resolution parameter $\dx \to 0$.   Even when the limiting solution is a map, the weak solution of \eqref{LPfull} can be a plan.  By using  barycentric projection of the transference plan, we recover the map, resulting in improved accuracy.

\subsection{Discretization: finite dimensional linear programming}
For claritiy and simplicity, we impose a Cartesian grid with uniform spacing $\dx$ on the domains $X$ and $Y$.   
The initial grids, $X^\dx$, and $Y^\dx$ are given by a  uniform Cartesian grid with spacing $h$ intersected with the support set, 
\[
X^\dx =  h\mathbb{Z^d} \cap X,
\qquad
Y^\dx = h\mathbb{Z^d} \cap Y.
\]
We enumerate the grid points, (which could also be referred to as \emph{quadrature} points), 
\[
X^\dx = \{x_1, \dots, x_n\},
\qquad
Y^\dx = \{y_1, \dots, y_m\}.
\]
Each $x_i$ is the center of a hypercube of width $\dx$,
\[
R_\dx(x_i) =  \left \{  x\in \R^d   \midB    \|x - x_i\|_\infty = \frac \dx 2  \right \}. 
\]
Define the approximate measures $\mu^\dx, \nu^\dx$, to be a weighted sums of Dirac masses whose weights corresponds to the integral of the measures over the hypercubes of width $\dx$ centred at $x_i$.  
\bq\label{discrete measures}
\begin{aligned}
\mu^\dx &= \sum_{i = 1}^n \mu_i^\dx \delta_{x_i}, 
\quad
\mu_i^\dx  = \mu( R_\dx(x_i) ) 
\\
\nu^\dx &= \sum_{i = 1}^m \nu_i^\dx \delta_{y_i}, 
\quad
\nu_i^\dx  = \nu( R_{\dx}(y_i) ) 
\end{aligned}
\eq
The discrete cost function is given by
\bq\label{discrete costs}
c_{ij} = c(x_i, y_j), 
\eq
With these definitions, the infinite dimensional optimal transportation problem \eqref{KantLP} is reduced to the finite dimensional linear programming problem \eqref{LPfull}.  We record it in the defintion which follows.

\begin{definition}\label{defn:2LP}
Define \eqref{LPfull}  at grid resolution $\dx$ to mean the costs and measures are given by  $c_{ij}^\dx, \mu^\dx, \nu^\dx$ using~\eqref{discrete measures} and~\eqref{discrete costs}.   
 Write $\pi^\dx = \pi^\dx_{ij}$ for an optimal  solution of  \eqref{LPfull}  at resolution $\dx$.
 From $ \pi^\dx_{ij}$, we recover the approximation to the optimal transference plan $\pi$ by the imbedding
\bq\label{discrete plan}
\pi^\dx = \sum_{i=1}^n \sum_{j=1}^m \mu^\dx_{ij} \delta_{(x_i, y_j)}.
\eq
\end{definition}

\begin{remark}
A more accurate method to discretize the measures would be use a centroidal Voronoi tessellation of the support of the measures, and let 
$x_i, y_j$ be the weighted barycenters of the Voronoi regions~\cite{doi:10.1137/S0036144599352836}, weights according to the measures $\mu$ and $\nu$, respectively.   In this case, we would refer to the points as \emph{quadrature points}.  This discretization of measures has been performed in \cite{merigot2011multiscale}. This corresponds to a \emph{quantization} problem, which is studied in the literature \cite{graf2000foundations}.  
\end{remark}

\subsection{Convergence of the linear progamming approximation}
In this subsection we prove that convergence of the solutions of \eqref{LPfull} to the solution of \eqref{KantLP} as $\dx \to 0$.   
We use a fundamental stability result for solutions of the optimal transportation problem,  along with the consistency of the approximation to obtain a weak convergence result.

Given $\mu, \nu$ and $c(x,y)$ as described in \S\ref{sec:Kant}.  Consider a sequence $\dx_k \to 0$.   
The following theorem is a special case of \cite[Theorem 5.20]{Villani2008Optimal}.
\begin{theorem}[Stability of optimal transportation]
Assume $X$ and $Y$ are compact, and that $c(x,y)$ is continuous.  
Let $\mu^{\dx_k}$ and $\nu^{\dx_k}$ be a sequence of measures that converge weakly to $\mu$ and $\nu$, respectively.  For each $k$ let $\pi^{\dx_k}$ be an optimal transference plan between $\mu^{\dx_k}$ and $\nu^{\dx_k}$.  Then, up to extraction of a subsequence,
\[
\pi^{\dx_k} \text{ convergences weakly to } \pi
\]
where $\pi$ is an optimal transference plan between $\mu$ and $\nu$. 
\end{theorem}

\begin{theorem}[Convergence of Linear Programming Solutions]
Suppose that  $X$ and $Y$ are compact, and that $c(x,y)$ is continuous.  
Let $\pi^\dx$ be a solution of \eqref{LPfull} 
given by the discretization \eqref{discrete measures} \eqref{discrete costs} \eqref{discrete plan}.
Then, up to extraction of a subsequence, $\pi^\dx$ convergences weakly to an optimal transference plan solution of \eqref{KantLP}.
\end{theorem}

\begin{proof}
The measures given by \eqref{discrete measures} are approximations by weighted Diracs masses over intervals whose volumes go to zero as $\dx \to 0$.  So clearly these measures converge weakly to the limits $\mu$ and $\nu$. 

The discrete cost function given by \eqref{discrete costs} can be extended to a continuous cost function defined on $X\times Y$ which converges uniformly to $c(x,y)$.   This can be accomplished, for example,  by  defining 
\[
c(x_i,y_j) = c^\dx_{ij}, \quad i=1,\dots, n, j = 1, \dots, m
\]
extend it continously to $X\times Y$, by multilinear interpolation. 

The solution of \eqref{LPfull} results in an optimal transference plan $\pi^\dx$ given by \eqref{discrete plan}.
Apply the stability theorem above  to obtain the desired weak convergence result.  \end{proof}

\subsection{Recovering the map from the plan using barycentric projection}\label{sec:Barycenter}
Even when the limiting transference plan $\pi$ is a map, 
the approximating  discrete marginal $\pi^\dx$ will not be map, in general.
However an approximation to the optimal map can be recovered from the transference plan using \emph{barycentric projection}.

The transference plan obtained from the linear programming solution $\pi^\dx$ is given by the imbedding  \eqref{discrete plan} 
$
\pi^\dx = \sum \pi^\dx_{ij} \delta_{(x_i, y_j)}.
$
Thus the  mass $\mu_i$ at $x_{i}$ is transported to multiple points $y_{j}$ with weights given by 
$
\mu_i x_i \mapsto \sum_{k=1}^m \pi_{ik} \delta_{y_k}.
$
\begin{definition}[Barycentric projection of transference plan]
Define for each $i$ with $\mu_i > 0$, $\bar y_i$ to be
the Euclidean barycenter of the points $(y_1,\dots, y_m)$ with weights $(\pi_{i1}, \dots \pi_{im})$, 
\[
\bar y_i = 
\frac{\sum_{k=1}^m \pi_{ik}y_{k}}{\sum_{k=1}^m\pi_{ik}}.
\]
Then define the barycentric projection of the transference plan $\pi^\dx$ by 
\[
\bar \pi^\dx =  \sum_{i=1}^n  \mu^\dx_i\delta_{(x_i, \bar y_i)}.
\] 
See Figure~\ref{fig:barycenter}.
\end{definition}

  See \cite[Definition 5.4.2]{ambrosio2006gradient} for the continuous case. 
Then $\bar \pi^\dx$ is a (discrete) \emph{map} which pushes forward $\mu^\dx$ to (an approximation) of $\nu^\dx$. 

\begin{remark}
Generally, $\bar y_{i}$ no longer belongs to the set $\left\{  y_{1},\dots,y_{m}\right\}$.  Correspondingly,   while the first marginal of $\bar \pi^\dx$ is equal to $\mu^\dx$, the second marginal of  $\bar \pi^\dx$  is generally different from $\nu^\dx$.
However, it is still a weak approximation of $\nu$, which converges weakly to $\nu$ as $\dx \to 0$.
\end{remark}
Barycentric projection is discussed in  \cite{ambrosio2006gradient}.  In particular, by Theorem 5.4.4 and by Lemma 12.2.3, we can conclude that, for convex costs, the barycentric projection of the approximations converges to the barycentric projection of the limiting transference plan.  In particular, when the unique limit is a map $\pi$, then $\bar \pi^\dx$ converges to $\pi$.

\begin{figure}[ptb]
\centering
\begin{minipage}{0.6\linewidth}
\includegraphics[width=\linewidth]{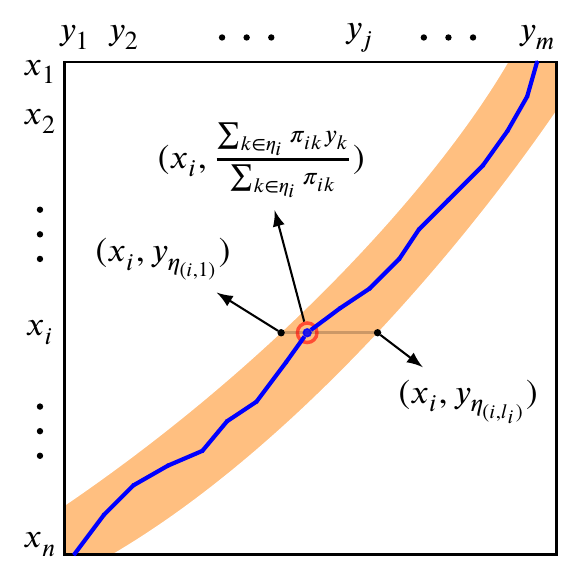}
\end{minipage}
\caption{Illustration of the  transference plan $
\pi = \pi^\dx_{ij}$ and the barycentric projection $\bar \pi$.
The coloured band represents the support of the transference plan.
Points in the same row are bundled into one as indicated by the small circle. The curve in the middle represents the resulting map.}%
\label{fig:barycenter}
\end{figure}

\section{Sparsity and a grid refinement procedure} \label{sec:refine}
The full linear program \eqref{LPfull} which approximates \eqref{KantLP} is too costly to solve for large problems.    In the case where the solution from \eqref{KantLP} is given by a map, we can expect that the discrete solution of \eqref{LPfull} will be \emph{sparse}, provided $\dx$ is small enough.   By estimating the  support of the transference plan, $\pi^\dx$,  the number of variables in \eqref{LPfull} is  reduced from quadratic to linear dependence on the inputs.  

Estimating the support in this manner forms the basis of the grid refinement procedure.
This process is iterated, leading to a multiple step grid refinement procedure. 

By considering a sequence of linear programs parameterized by $\dx$ we can find the support of the solution,  $\pi^\dx$,  from the support of the solution corresponding to a smaller sized problem.  
Then we can  solve the reduced size linear program, \eqref{LPsparse} below, with the expectation that the solution of the sparse problem is the same as that of the full problem.   This will be the case if the previously estimated support is exact.   We make a heurstic argument for why these supports should be the same, appealing to the stability property of linear programs under perturbation.  But this argument is not rigorous in the limit $\dx\to 0$, since the sensitivity of the linear program can depend on $\dx$.   However, by growing the support we can test after each step if the support exceeds the estimated support.  In practice, this has only happened for a the first stage for very small initial grid sizes.   However, if this were to occur, the algorithm could be modified as needed.

An alternative is to use the method of Schmitzer \cite{schmitzer2015sparse} (see also \cite{schmitzer2013hierarchical}), which ensures global optimality of the solution to the sparse subproblem by 
using locally adapted neighborhood sizes.     

We mention another apsect of the refinement technique which could be improved.   In the interested of keeping  the already complex code simpler, we refined using bisection in each square.  An improvement in accuracy, and perhaps also the sensitivity of the refined problem would be to refine adaptively, either according to the densities, or according to the senstitivity of the linear program.  In the latter case, we could also expect an inprovement in the accuracy.

\subsection{The sparse linear program}
Let $\pi = \pi_{ij}$ be the solution of \eqref{LPfull}, and let $\Ind_0$ be the basis set, the indices of the nonzero entries of $\pi$.     
If  $
\Ind_0 \subset \Ind$, for a known set, $\Ind$, then we can recover $\pi$ by solving the reduced linear program, 
\bq \label{LPsparse} \tag{LPR}
\begin{aligned}
\text{Minimize   } & \sum_{(i,j) \in K }  c_{ij} \pi_{ij}, 
\\ \text{ Subject to: } 
&
\left \{
\begin{aligned}
\sum_{ \{ j \mid (i,j) \in \Ind \}  }  \pi_{ij} &=\mu_i, & i\in [1,\dots,n]
\\ 
\sum_{ \{ i \mid (i,j) \in \Ind \}  }  \pi_{ij}  &=\nu_j, & j \in [1,\dots,m]
\\ 
\pi_{ij} &\ge 0,  & (i,j) \in \Ind 
\end{aligned} 
\right .
\end{aligned}
\eq

\begin{remark}
Note that, in contrast to~\eqref{LPfull}, the number of variables in \eqref{LPsparse} is $\abs{K}$, and the number of constraints is $\abs{K} + n + m$.    In practice, $\abs{K}$  is a small multiple of $n+m$.  
So the size of the reduced linear programming problem~\eqref{LPsparse} grows linearly in the number of variables used to represent the measures. 
\end{remark}

\subsection{Multiscale solution procedure}
The multiscale solution procedure is described here.  Refer to Figure~\ref{fig:Refinement}, which illustrates the process for a one dimensional example.

\begin{figure}[ptb]
\centering
\includegraphics[width=0.66\linewidth]{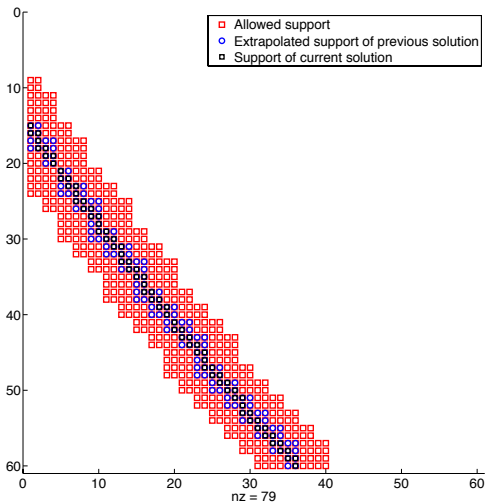}
\vspace{-.3in}
\caption{Illustration of the multiscale solution procedure.
}%
\label{fig:Refinement}%
\end{figure}

\begin{enumerate}
\item[(0)] \label{step:start}  The full linear program \eqref{LPfull}, with $c_{ij} = c_{ij}^\dx$ and $\mu = \mu^\dx, \nu = \nu^\dx$ as given by the quadrature rules~\eqref{discrete measures} and~\eqref{discrete costs},   is solved on a coarse grid, $X^\dx\times Y^\dx$.   

\item\label{grow support}  
From the discrete solution, $\pi^{\dx}_{ij}$,  recover  the spatial support set,
\[
S_\dx  = \left \{ R_\dx(x^\dx_i)\times R_\dx (y^\dx_j)    \midB  \pi^{\dx}_{ij}  > 0 \right \},
\qquad S_\dx \subset X^\dx\times Y^\dx.
\]
Grow the spatial support set, by including allowing neighbours (in space),
\[
\bar  S_{\dx} = \{ \text{ neighbours } S_\dx  \}.
\]
Extract the corresponding indices
\[
\overline \Ind_\dx=  \{ \text{ indices of center points in  } \bar S_\dx  \}.
\]

\item Refine the grids in each domain by a factor of two, labelling the new grids $X^{\dx/2}$,  $Y^{\dx/2}$.  
In each coordinate of $x = (x_1, \dots, x_d)$, the interval 
$[x_i -\dx, x_i + \dx]$ is halved,  to become $[x_i -\dx, x_i], [x_i,x_i + \dx]$, 
and the two new points are generated at the midpoints of the interval.  Thus,  the hypercube $R_\dx(x)$ is divided evenly into $2^d$ hypercubes, and the grid point $x$ generates $2^d$ new points on the refined grid, each at the centres of the smaller hypercubes.  

\item \label{step:refine support} The allowed spatial support is interpolated onto the finer grid $S_{\dx/2} = \bar S_{\dx}$
and the corresponding indices are extracted
\[
\Ind_{\dx/2} = \{ \text{ child indices of } \overline \Ind_\dx  \}.
\]

\item\label{step: refine data}  The sparsely supported linear programming problem \eqref{LPsparse} is solved with indices $\Ind_{\dx/2}$ and with $c^{\dx/2}$ and $\mu^{\dx/2}, \nu^{\dx/2}$ given by the quadrature rules~\eqref{discrete measures} and~\eqref{discrete costs}. 

\item Repeat starting at \eqref{grow support}, until a fine enough solution is computed.
\end{enumerate}

\begin{remark} Depending on the coarseness of the approximations, Step \eqref{grow support}  can be skipped, or repeated, so that neighbours of neighbours are included. The computational cost of adding additional support indices is not significant.  

Step \eqref{step:refine support} is geometrically simple;  algorithmically, an indexing formula is used to determine the child indices.

In step \eqref{step: refine data} \eqref{LPsparse}  has $\abs{\Ind_{\dx/2}}$ variables, which is on the order of $2^d (n + m)$. 
\end{remark}

\subsection{Justification of the grid refinement procedure}\label{sec:proofGrid}
In the case where the continuous plan $\pi$ is a map, from the convergence result of the linear programming approximation, we know that as $\dx \to 0$, the support of the discrete marginal $\pi^\dx$ is a weak approximation of the support of $\pi$.    Consequently, for $x_i, y_j$ away from the support of $\pi$, $\pi^\dx_{x_i,y_j} = 0$ is zero, for $\dx$ small enough.   

In practice, we found that the support of the refined solution was always contained in $\Ind_{\dx/2}$.  
Should the property ever fail, the code  checks if the support $\pi^\dx$ reaches the boundary of the index set, $\Ind_{\dx}$.  

In this section, we give an indication of when the estimation of the support given by $\Ind_{\dx/2}$ is exact for the refined solution.  

Consider the  standard form linear program,
\bq\label{LPstandard}
\begin{aligned}
\text{Minimize   } &  c^\intercal x 
\\ \text{ Subject to: } 
&
\left \{
\begin{aligned}
A x &= b
\\
x &\ge 0.
\end{aligned} 
\right .
\end{aligned}
\eq
Then we have the  following stability result \cite[Chapter 5]{bertsimas1997introduction}.
\begin{theorem}[Stability of linear programming]
If the standard form linear program \eqref{LPstandard} has a unique optimal solution $x^*$, then the support of the optimal solution $x^*(c)$ is unchanged for nearby values $c$. 
If it has  has a nondegenerate  optimal basic solution $x^*$ , then the support of the optimal solution $x^*(b)$ is unchanged for nearby values $b$. 
\end{theorem}
In this context, a basic solution means the support is equal to the number of rows of $A$, which is assumed to be full row rank.  This corresponds to $n + m-1$ nonzeros for \eqref{LPfull}. 

Consider the following  linear programming problems.
\begin{definition}\label{defn:2LP}
Define \eqref{LPfull} or \eqref{LPsparse} at grid resolution $\dx$ to mean the costs and measures are given by 
$c_{ij}^\dx, \mu^\dx, \nu^\dx$ using~\eqref{discrete measures} and~\eqref{discrete costs}.   
 Write $\pi^\dx$, $\pi^{\dx/2}$ for an optimal  solution of  \eqref{LPfull}  at resolution $\dx$, $\dx/2$, respectively.   Write $\pi^{\dx/2}_\Ind$ for the solution of \eqref{LPsparse} at resolution $\dx/2$ using the indices $\Ind = \Ind_{\dx/2}$ given by support estimation procedure, as in Step~\eqref{step:refine support}.
\end{definition}

In the following lemma we show that there is a linear program with data close to that for \eqref{LPfull}  at scale $\dx/2$, which has the correct support $\Ind_{\dx/2}$ obtained from~$\pi^\dx$. 
\begin{lemma}
Consider the linear programming solution $\pi^{\dx/2}$.  There is a perturbation of the cost function 
$\hat  c^{\dx/2}$, for which the support of the solution, $\hat \pi^{\dx/2}$,  of the perturbed problem is contained in the indices $\Ind_{\dx/2}$ obtained using Step~\eqref{step:refine support} from the coarse solution $\pi^\dx$.
The perturbed cost can be made arbitrarily close to $c^{\dx/2}$ by taking $\dx\to 0$.
\end{lemma}
%

\begin{proof}
Define an auxiliary  linear program from \eqref{LPfull} at scale $\dx$ as follows. 
On the grid of scale $\dx$, each grid point $x_i$ or $y_j$ is split into $2^d$ children.  Denote an arbitrary child of $x_i$ by $\hat x_i$, and similarly for $y_j$.   Define the ``doubled''  linear program \eqref{LPfull}, on the grid at scale $\dx/2$,  as follows.  The measures are still given by $\mu^{\dx/2}$ and $\nu^{\dx/2}$.  
The cost function is given by 
\[
\hat c_{\hat x_i, \hat y_j}^{\dx/2} = c_{x_i, y_j}^\dx,
\qquad
\text{ for any child grid points $\hat x_i, \hat y_j$ of $x_i, y_j$}
\]
Observe from the definition \eqref{discrete measures}, and from the definition of the child points $\hat x_i$ of $x_i$, that 
\bq\label{mass children}
\mu_{x_i}^\dx = \sum_{\hat x_i}  \mu^{\dx/2}_{\hat x_i},
\qquad
\nu_{y_i}^\dx = \sum_{\hat y_i}  \nu^{\dx/2}_{\hat y_i}.
\eq

Now having defined \eqref{LPfull} with the data above, 
let $\hat \pi^{\dx/2}$ be a solution.   Note that there are multiple solutions, since pairs of child vertices have equal costs.  
Recall that $\pi^\dx$ from Definition~\ref{defn:2LP} is the solution of \eqref{LPfull} at scale $\dx$. 
Observe that  $\hat \pi^{\dx/2}$ is a splitting of  $\pi^\dx$, in the sense that 
\[
\pi^\dx_{x_i, y_j}  = \sum_{\hat x_i, \hat y_j} \hat\pi^{\dx/2}_{\hat x_i \hat y_j}  
\]
where the sum is over all child grid points $\hat x_i$ of $x_i$ and $ \hat y_j$ of $y_j$.  This follows from the fact that the cost $\hat c$ is constant over pairs of child grid points, and also, since the mass of the marginals is split amongst the children \eqref{mass children}.

So the solution of the linear program, $\hat \pi^{\dx/2}$,  corresponds to a redistribution of the marginal $\pi^{\dx}$ to to child vertices based on the refinement of the measures.
This means that the support of $\hat \pi^{\dx/2}$ is contained in $\Ind_{\dx/2}$ from Step~\eqref{step:refine support}.  (It could be smaller, which could arise, for example, if some $\mu_{\hat x_i} = 0$.)

Next observe  that \eqref{LPfull} at resolution $\dx/2$ and the ``doubled'' linear program  are in the standard form \eqref{LPstandard} with the same values for $A$ and $b$, but different values of~$c$. 

Recalling the definition of the costs,~\eqref{discrete costs}, at any point $x_i, y_j$, they differ by at most $c(x_i, y_j) - c(x_i + x, y_j + y)$ where $\abs{x}_\infty, \abs{y}_\infty \le \dx/2$.  Since by assumption the cost function is continuous, this difference can be made arbitrarily small. 
\end{proof}

By the stability theorem for linear programming, if $\pi^{\dx/2}$ is the unique optimal solution,  of \eqref{LPfull} at scale $\dx/2$, and if $\hat c^{\dx/2}$ is close enough to $c^{\dx/2}$, then the support of $\pi^{\dx/2}$ is the same as the support of $\hat \pi^{\dx/2}$, which  is contained in $\Ind_{\dx/2}$.   So the sparse linear programming solution $\pi^{\dx/2}_\Ind$ has the same support as $\hat \pi^{\dx/2}$, and, since the data is the same for  $\pi^{\dx/2}_\Ind$ and  $\pi^{\dx/2}$, 
\[
\pi^{\dx/2}_\Ind= \pi^{\dx/2}. 
\]
To summarize,  in the case where we can apply the stability theorem, as above, we can conclude that the solution of the sparse linear program \eqref{LPsparse} with support $\Ind_{\dx/2}$  is equal to the solution of the full problem \eqref{LPfull} at resolution $\dx/2$.

\section{Numerical results}\label{sec:numerics}

\subsection{Performance}\label{sec:performance}
We used MATLAB to generate the full and sparse LP problems and to call the LP solver.  We tried several LP solvers, and found that several commercial products (MOSEK, Gurobi, CLPEX) performed equally well.  We called these products from CVX, and also called them directly from MATLAB.

For benchmarking performance we used the example  of transporting uniform densities from a square to diamond, see Figure~\ref{fig:CompareCircles}.  The LPs were generated in MATLAB using Gurobi as a solver.   Table~\ref{table:OTPerformance} compares computation time and memory usage of the multigrid LP solver the full LP solver and the entropy method, implemented using sample code provided by the authors of \cite{benamou2015iterative}.  Computation time and memory usage grow linearly with the number of unknowns, $2N^{2}$, for the Linear Program, as shown in Figure~\ref{fig:OTPerformance}. We used a PC laptop (i3 1.9GHz CPU with 12GB RAM).  For comparison, we also ran the example with $N=512$ on a 2015 MacBook Air (2.2GHz intel i7 CPU with 8GB RAM) and the run time decreased by about 30\%.

For the full solver,  the memory usage with $N = 64$ or $8192$ variables was 7 GB and the computation time was nearly 527 seconds.  In comparison, the multigrid method was an improvement in both time and memory  of two orders of magnitude. The largest sized problem we computed using the multiscale solver corresponded to $N=512$, or about half a million variables.  Compared to the largest problem for the full solver, this is an increase of problem size two orders of magnitude.   The performance of the entropy method was faster than the full solver, by about one order of magnitude, and allowed for a problem size of about double the maximum problem size for the full LP solver.  Compared to the entropy method at the largest problem size, $N=96$, or $18432$ variables, the multiscale solver was about 50 times faster, and used about 40 times less memory.
A comparable run time corresponded to a problem size of $N=512$ or $524288$ variables.
The method in  \cite{solomon2015convolutional} introduced performance improvements over the implementation of \cite{benamou2015iterative} as well as hardware improvement.  They used a computer with 23.5GB RAM, implemented on a GPU, (which is significantly faster than a CPU), the run times reported here, which were limited to smaller problem sizes, were comparable to ours.


\begin{table}[h]
\centering
\begin{tabular}{rrrrrrrr}
\multirow{2}{*}{N (grid)} & \multirow{2}{*}{$2N^2$} & \multicolumn{2}{c}{MGLP} & \multicolumn{2}{c}{FLP} & \multicolumn{2}{c}{ER} \\
 &  & CPU & Memory & CPU & Memory & CPU & Memory \\
32 & 2048 & 0.9 & 14 		& 14.0 & 400 & 3.8 & 900 \\
48 & 4608 & 1.6 & 40 		& 95.6 & 2000 & 20.1 & 1200 \\
64 & 8192 & 3.2 & 80 		& 527.1 & 7000 & 64.4 & 1900 \\
96 & 18432 & 7.0 			& 160  &$*$  & $*$ 	& 310.1 & 6200 \\
128 & 32768 & 13.5 & 300 		&      	& 		& $*$ & $*$ \\
192 & 73728 & 35.3 & 700 	&      	&   &      	& \\
256 & 131072 & 58.9 & 1100 	&    	&   &      	&\\
384 & 294912 & 165.5 & 2500 	&   	&   &      	&\\
512 & 524288 & 287.6 & 4000 	&   	&   &      	&
\end{tabular}
\caption{
Comparison of run time and memory usage for the multigrid LP (MGLP), the full LP (FLP), and the entropic regularization (ER) method, the latter with precision to $10^{-4}$ and regularization parameter $\epsilon = 10^{-3}$ (\cite{benamou2015iterative}). Precision for the LP solvers are both $10^{-8}$. Memory usage for the LP solvers are as reported by Gurobi, and is estimated for the ER method.
}
\label{table:OTPerformance}
\end{table}

\begin{figure}[ptb]
\centering
\begin{minipage}{0.5\linewidth}
\includegraphics[width=\linewidth]{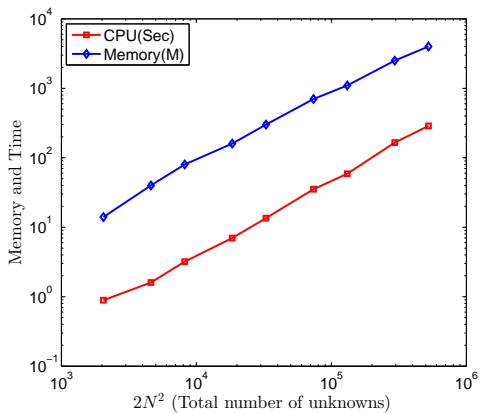}
\end{minipage}
\caption{CPU time and memory usage of the multiscale linear programming
solver. Both memory usage and computation time scale linearly with problem size.}%
\label{fig:OTPerformance}%
\end{figure}

\subsection{Smoothing the density}
The target density obtained by the push forward of the barycentric projection of $\pi^\dx$ is still given as a weighted sum of Dirac masses.    We can improve the accuracy of the result by smoothing the density, in other words by simply replacing each Dirac mass by a sum of Gaussians with standard deviation on the same scale as $\dx$.  We took values between 5 and 9 multiples of $\dx$.     More general methods for Kernel Density Estimation in the context of statistical inference are discussed in~\cite{silverman1986density}.

Figure~\ref{fig:CompareCircles} is a comparison of optimal transport
map before and after applying Gaussian filter using the map between uniform densities on a square and a diamond, discussed further in \S\ref{ssec:OTnum}.  The second part of the figure shows smoothing of a computed  barycenter, see \S\ref{sec:barcyenterN}.

\begin{figure}[ptb]
\centering\begin{minipage}{0.8\linewidth}
\includegraphics[width=0.5\linewidth]{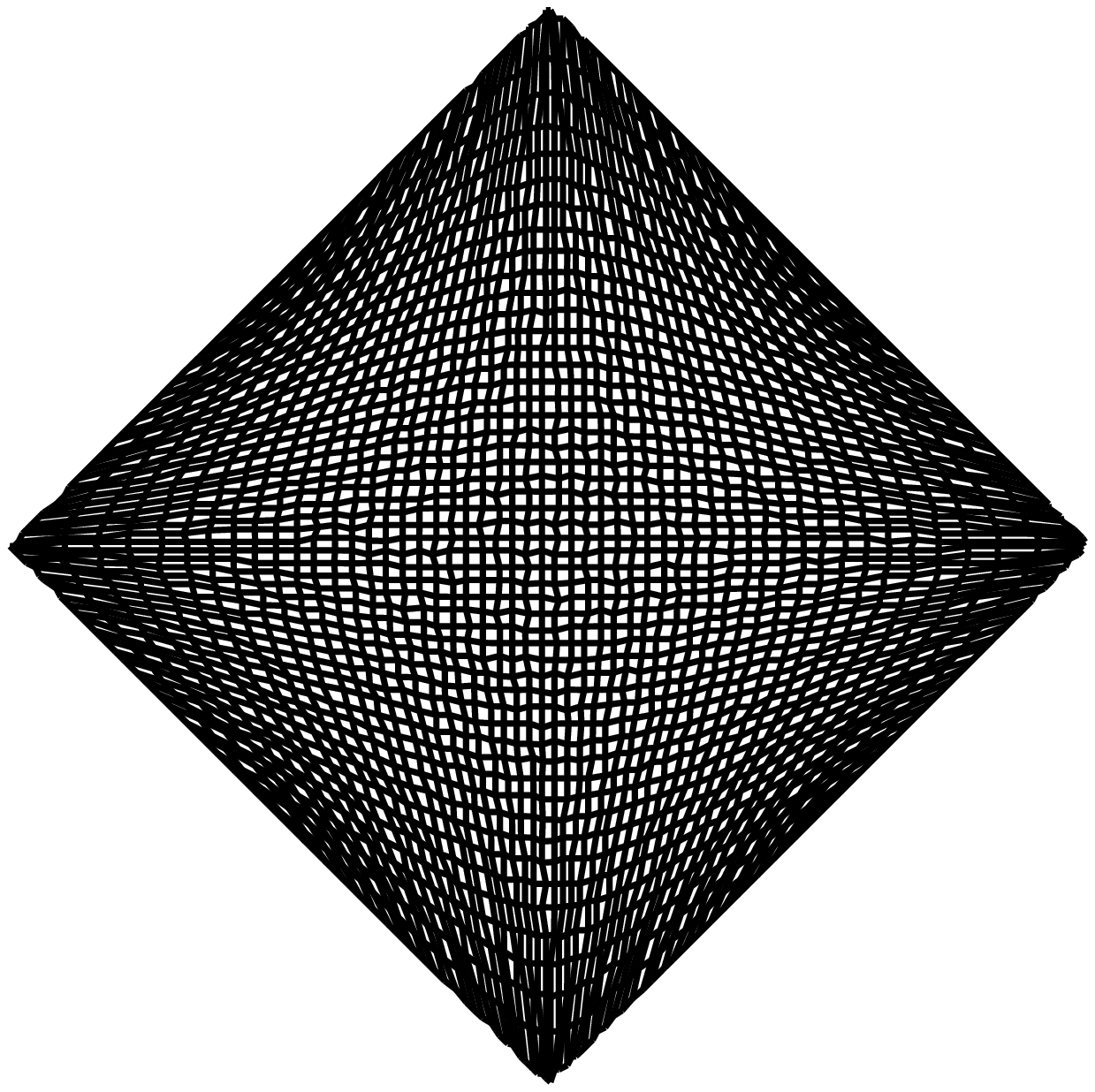}
\quad
\includegraphics[width=0.5\linewidth]{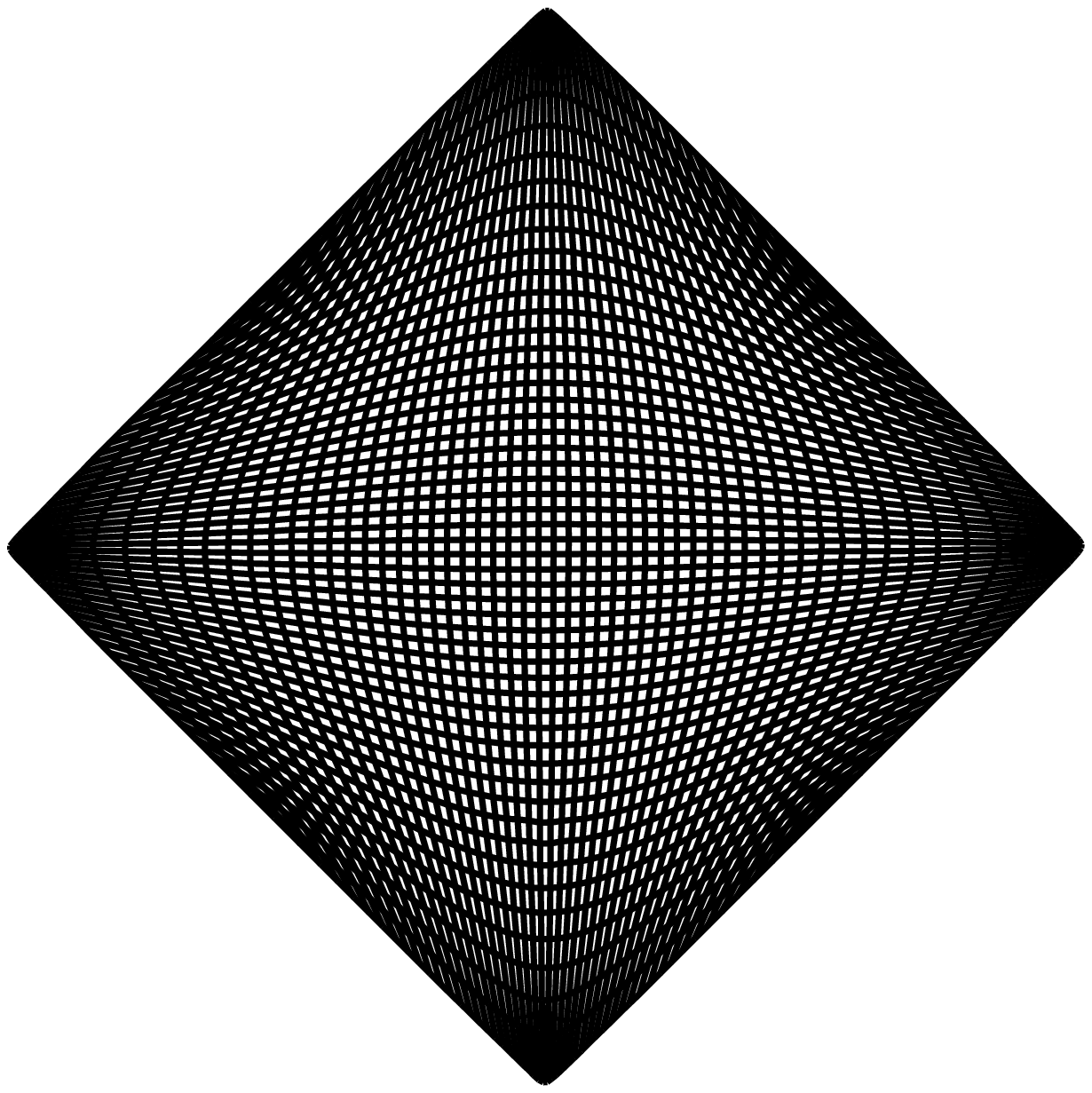}
\end{minipage}
\begin{minipage}{\linewidth}
\includegraphics[width=0.5\linewidth]{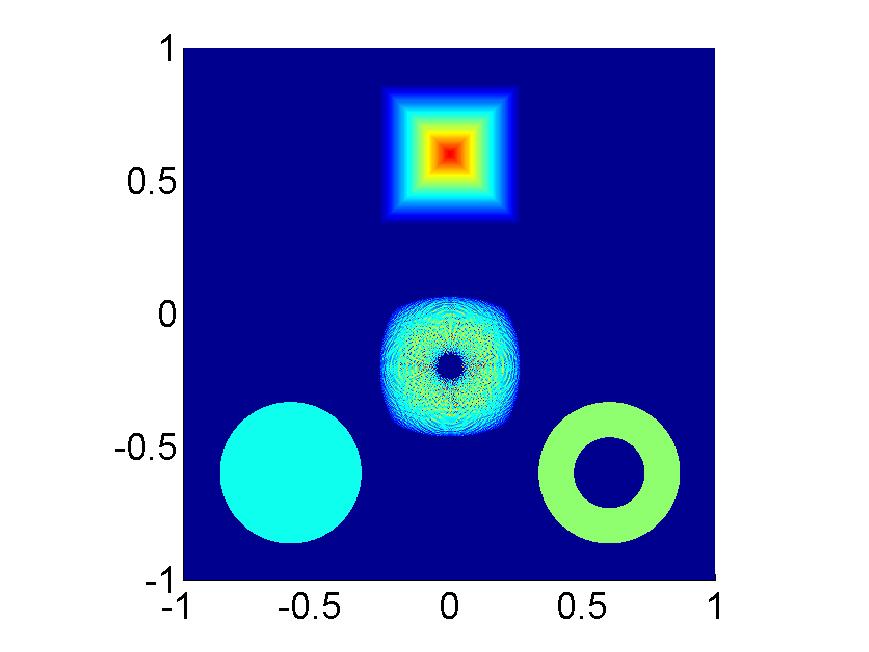}
\includegraphics[width=0.5\linewidth]{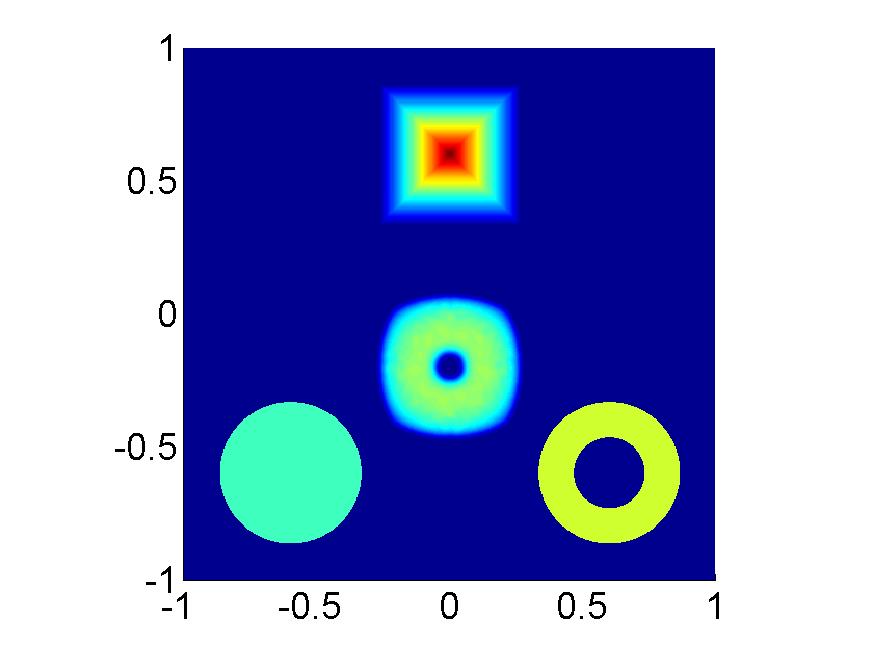}
\end{minipage}
\caption{Top: Optimal transport map from square to diamond recovered
by barycentric projection. Left: Before smoothing. Right: After smoothing.
Bottom: Barycenter of a square and $2$ Cylinders. Left: Before smoothing.
Right: After smoothing.
}%
\label{fig:CompareCircles}%
\end{figure}\ \ \ \ \ \ \

\subsection{Accuracy of the solutions}\label{sec:accuracy}
In this section we report on  the accuracy of solutions 
 in terms of the grid spacing $\dx$, by considering exact solutions to the Monge-Amp\`ere PDE.  These examples were used in~\cite{Benamou2014107}.
The first example is for a variable density.  The second example the well-known example of Caffarelli which splits a circle.  In the case the mapping consists of two translations. In these examples, both maximum and $L_{2}$ error decrease approximately linearly with grid size  which is comparable with the results in~\cite{Benamou2014107}.  The solutions are plotted in Figure~\ref{fig:Comparison}.

\begin{example}\label{ex:33}
\label{ExSquareWithDensity}Transportation between two squares of the same
size. The target has constant density, while the source square (Figure~\ref{fig:Comparison}) has a density given by
\[
f\left(  x,y\right)  =1+4\left(  q^{\prime\prime}(x)q(y)+q(x)q^{\prime\prime
}(y)\right)  +16\left(  q(x)q(y)q^{\prime\prime}(x)q^{\prime\prime
}(y)-q^{\prime}(x)^{2}q^{\prime}(y)^{2}\right)  ,
\]
where
\[
q(z)= \left (-\frac{1}{8\pi}z^{2}+\frac{1}{256\pi^{3}}+\frac{1}{32\pi} \right)
\cos(8\pi z)+\frac{1}{32\pi^{2}}z\sin(8\pi z);
\]
The optimal transportation map is given by%
\[
\left\{
\begin{array}
[c]{c}%
u_{x}\left(  x,y\right)  =x+4q^{\prime}(x)q(y)\\
u_{y}\left(  x,y\right)  =y+4q(x)q^{\prime}(y)
\end{array}
\right.
\]
The accuracy of the solution is presented in Table~\ref{table:CompareSquaresTable}.
\end{example}

\begin{example}\label{ex:34}
\label{ExSplittingCircle}Figure \ref{fig:Comparison} shows a circle is
split into two half circles. The potential is given by
\[
V(x_{1},x_{2})=\frac{1}{2}(\left\vert x_{1}\right\vert +x_{1}^{2}+x_{2}^{2});
\]
The accuracy of the solution is presented in Table~\ref{table:CompareSquaresTable}.
\end{example}

\begin{table}[h]%
\begin{tabular}
[c]{rlllll}%
$N$ (gridsize) & 32 & 64 & 128 & 256 & 512\\\hline
Max error Ex~\ref{ex:33} & 0.00721 & 0.00892 & 0.00689 & 0.00241 & 0.00148\\
$L_{2}$ error Ex~\ref{ex:33} & 0.00385 & 0.00379 & 0.00257 & 0.00103 & 0.00047\\
Max error Ex~\ref{ex:34} & 0.0625 & 0.03125 & 0.01563 & 0.00781 & 0.00391\\
$L_{2}$ error Ex~\ref{ex:34} & 0.01211 & 0.00302 & 0.00148 & 0.00050 & 0.00018
\end{tabular}
\caption{Accuracy of the solution for Example~\ref{ex:33} and~\ref{ex:34}.}
\label{table:CompareSquaresTable}
\end{table}

\begin{figure}[ptb]
\centering
\begin{minipage}{0.32\linewidth}
\includegraphics[width=\linewidth]{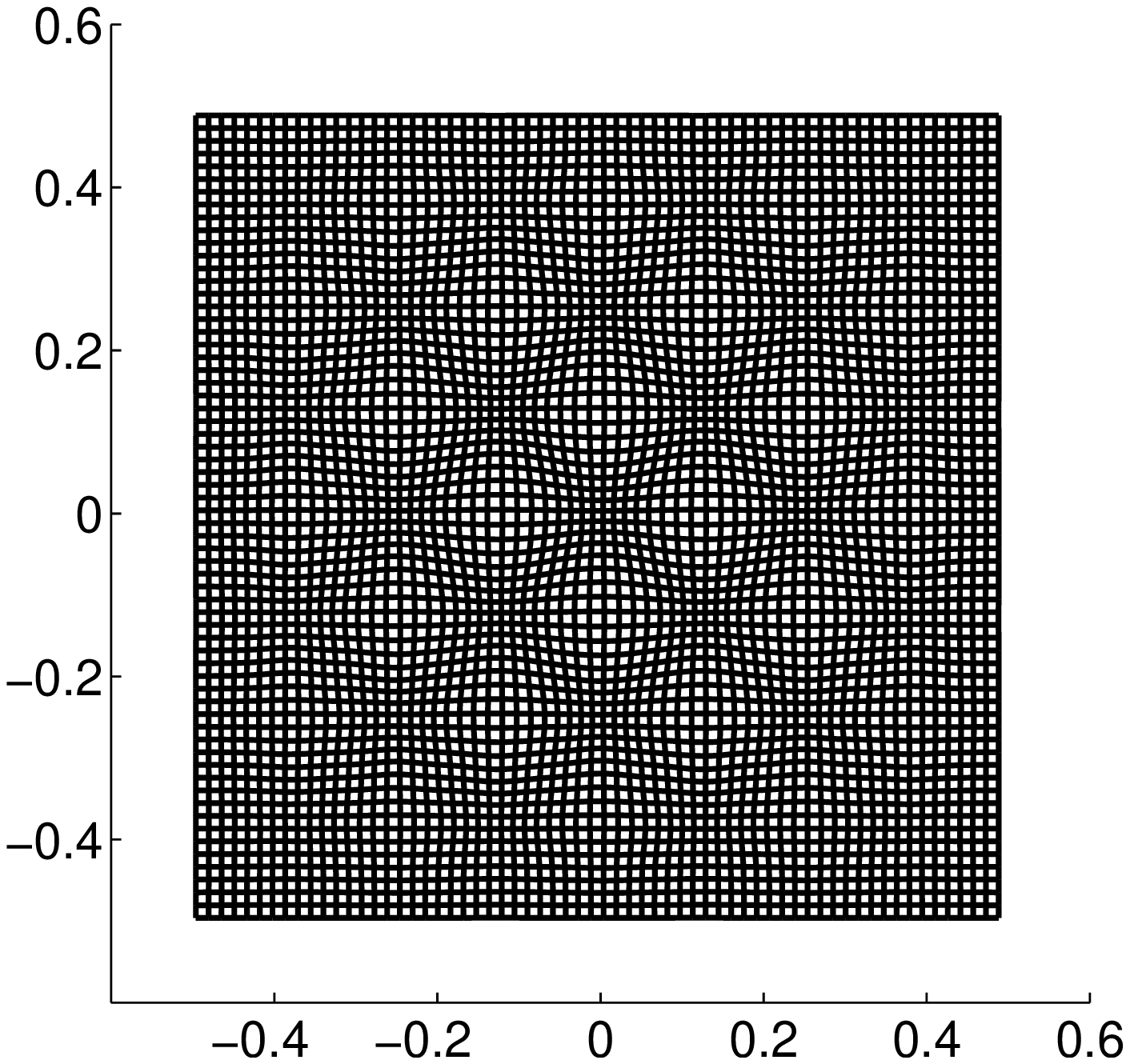}
\end{minipage}
\quad\begin{minipage}{0.64\linewidth}
\includegraphics[width=.5\linewidth]{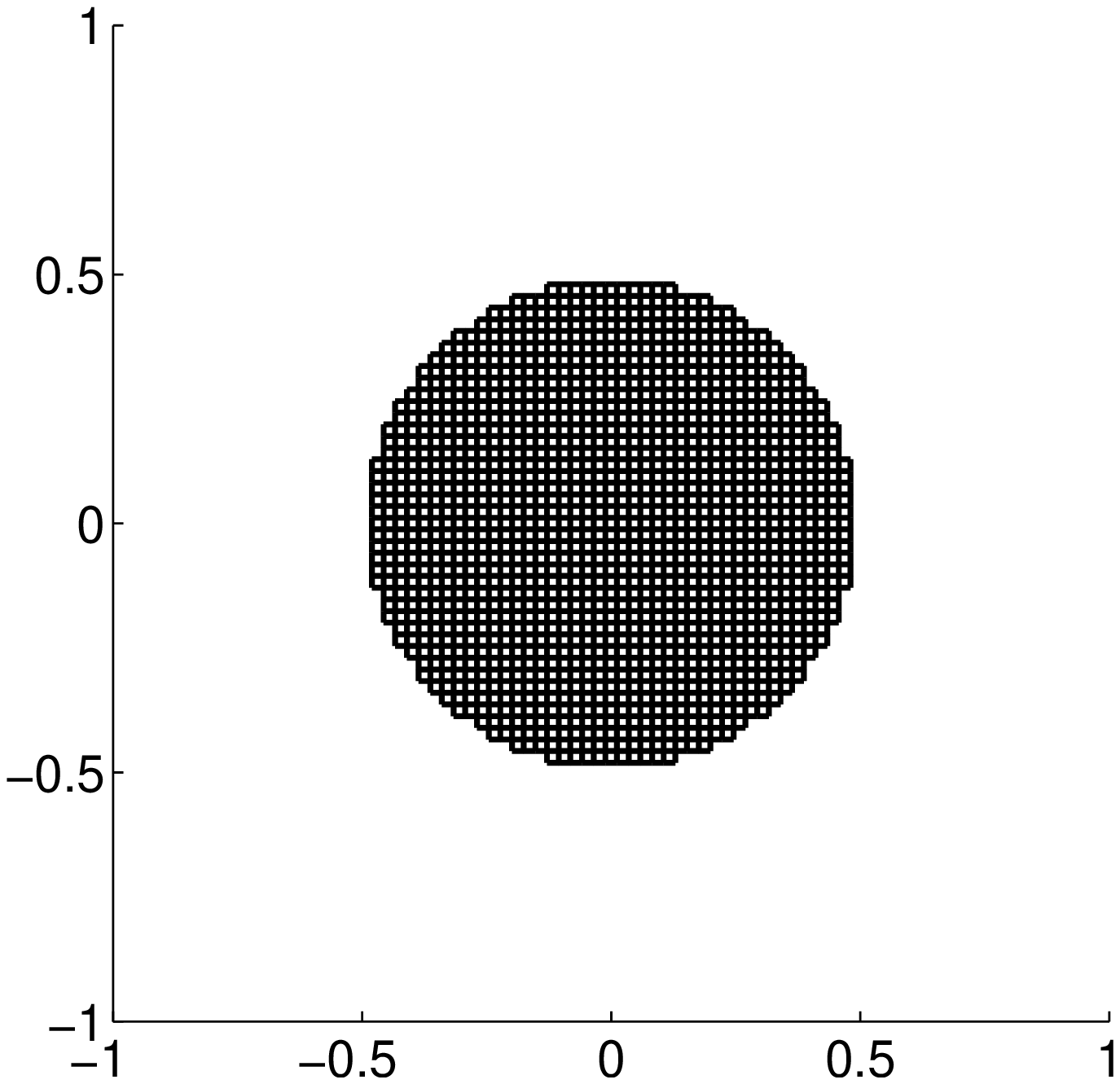}
\nolinebreak%
\includegraphics[width=.5\linewidth]{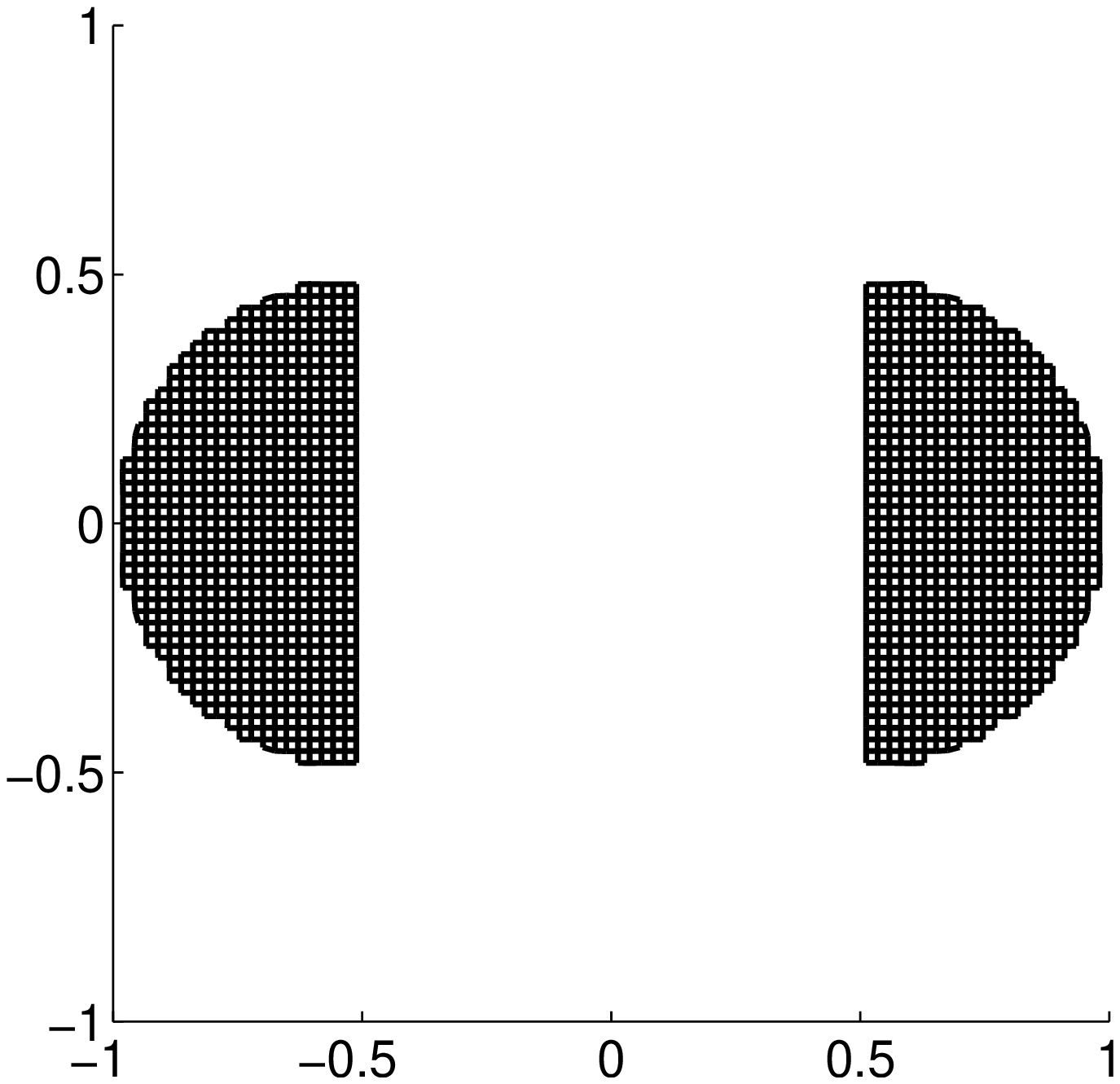}
\end{minipage}
\caption{Optimal transport map recovered by barycentric projection. Left:
Target of Example \ref{ExSquareWithDensity}. Middle: Source of Example
\ref{ExSplittingCircle}. Right: Target of Example \ref{ExSplittingCircle}}%
\label{fig:Comparison}%
\end{figure}

\subsection{Numerical solutions of Optimal Transportation Problem}\label{ssec:OTnum}
Visualizing OT maps requires care.  Our strategy for visualizing maps begins, following \cite{Benamou2014107}, by taking the source domain to be a square with uniform density,  and by pushing forward the image of the grid lines onto the target domain.   Subsequently, we consider more complicated source domains, still mapping horizontal and vertical lines from the target.   We also label points in the source as well as their images for clarification. 
We suggest spending time  inspecting the images. 

We begin by computing examples of for the Optimal Transportation problem with quadratic cost $c(x,y) = \abs{x-y}^2$.   With restrictions on the measures explained in Example~\ref{ex:MongeAmpere}, the solution is given by solving the Monge-Amp\`ere Partial Differential Equation. These maps were computed in \cite{Benamou2014107}, and we refer to this paper for visualizations of maps to convex targets.   The first examples are uniform densities.  Example~\ref{exOT1} shows the optimal  map from a square to a non-convex target.  
Example~\ref{exOT11} shows the optimal map from a square to a diamond. 
Example~\ref{exOT2} shows a map between two non-convex domains.  

In Example~\ref{exOT3} , we consider costs of the form $c(x,y) = \abs{x-y}^p$, and consider  a piecewise constant  density on the source rectangle.

\begin{example}\label{exOT1}
First consider the OT problem with $\mu$ the uniform density on a square and $\nu$  to the uniform density on the  ``Pac-Man'', a  circle with a section of $\pi/2$ radians removed.  The mapping is presented in~Figure~\ref{fig:pacman}.
The map is discontinuous on a line segment market $OG$ on the square, which  is mapped to two lines in the target, from the center $O$ to halfway
down the radial segments.   The map is continuous on the rest of the domain.
Near the points $E,F$ in the source, vertical lines of
the form $x=c$ are mapped to corners, lines of the form $x=|y|+c$.
\end{example}

\begin{example}\label{exOT11}
Figure~\ref{fig:CompareCircles} shows the optimal map from $\mu$ the uniform density on a  square, $X$,  to $\mu$, the uniform density on a  diamond.  The source square is not plotted, but it can be seen in~Figure~\ref{fig:pacman}.  The map shares the symmetries of the target.  But while the central horizontal and vertical lines are mapped to straight lines, lines on the boundary of the square bend to the  corners.  Squares near the center of $X$  are mapped to squares, while squares need the edges of $X$ are mapped to rectangular shapes with large aspect ratios. 
\end{example}

\begin{example}\label{exOT2}
Refer to Figure~\ref{fig:pacman}.
The source domain, $X$ is given by a square with a medium  square removed from  the center, and with smaller squares removed from the four corners.
The target domain, $Y$,  is the set $\{ \abs{x}^{4/5} + \abs{y}^{4/5} \le 1 \}$ with an inner diamond shape removed, given by  $\{ \abs{x} + \abs{y} \le 9/25 \}$.

We mapped the vertical and horizontal lines in the source domain to the target domain.   The problem was solved on a $512^{2}$ grid, but plotted on a $256^{2}$ grid for clarity.  We indicated the images of eight points, the inner and outer diagonal corners of $X$.   Both  shapes are  symmetric about diagonal lines, and this symmetry is preserved by the map.  The inner points in the target domain show discontinuities in the inverse map; points from far away in the source are mapped to points near, for example,  $A^\prime$.  The dent near $A^\prime$ indicates the discontinuity: as more lines are used in the plot  the  size of the dent decreases.
\end{example}

\begin{figure}[tb]
\includegraphics[width=1.2\linewidth]{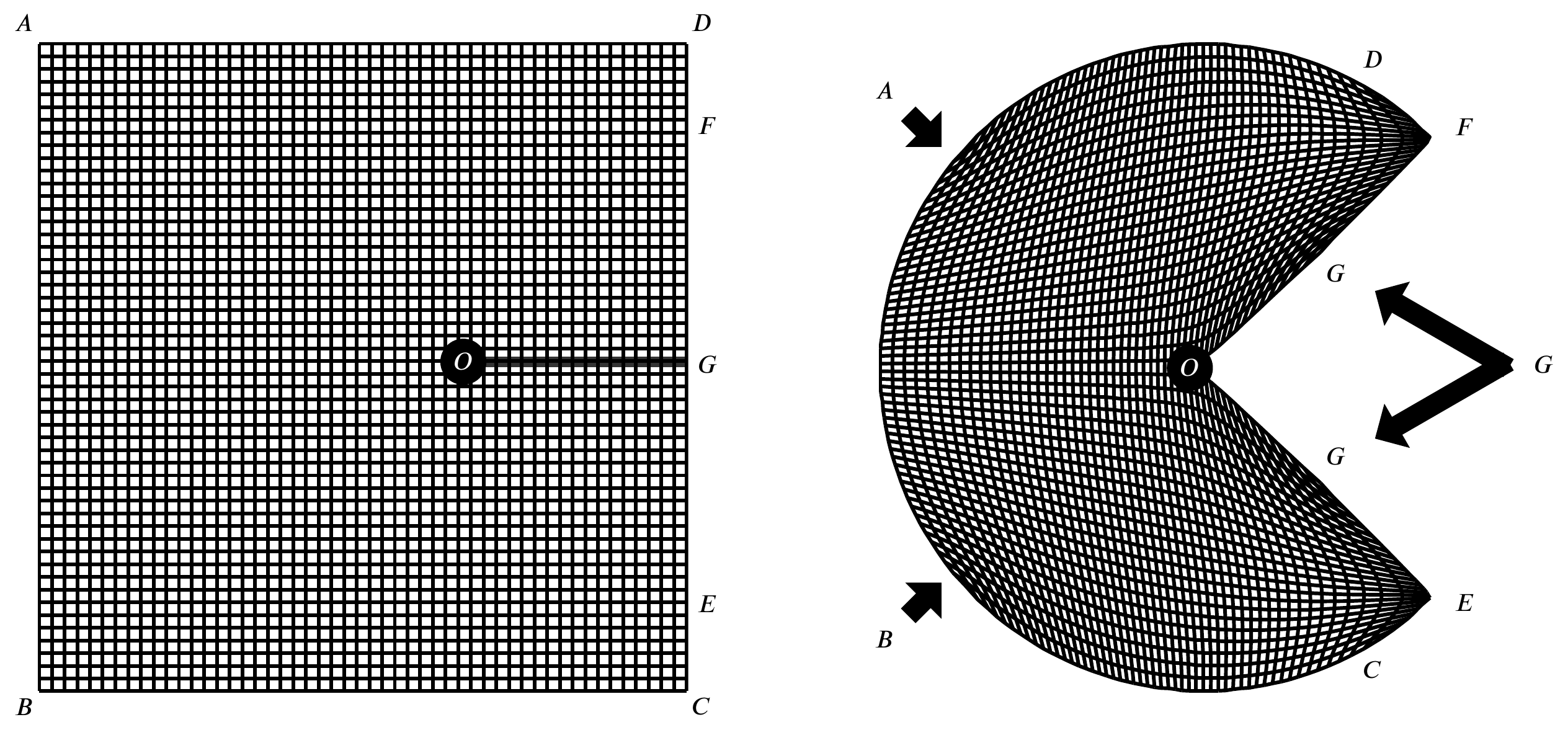}
\includegraphics[width=1.2\linewidth]{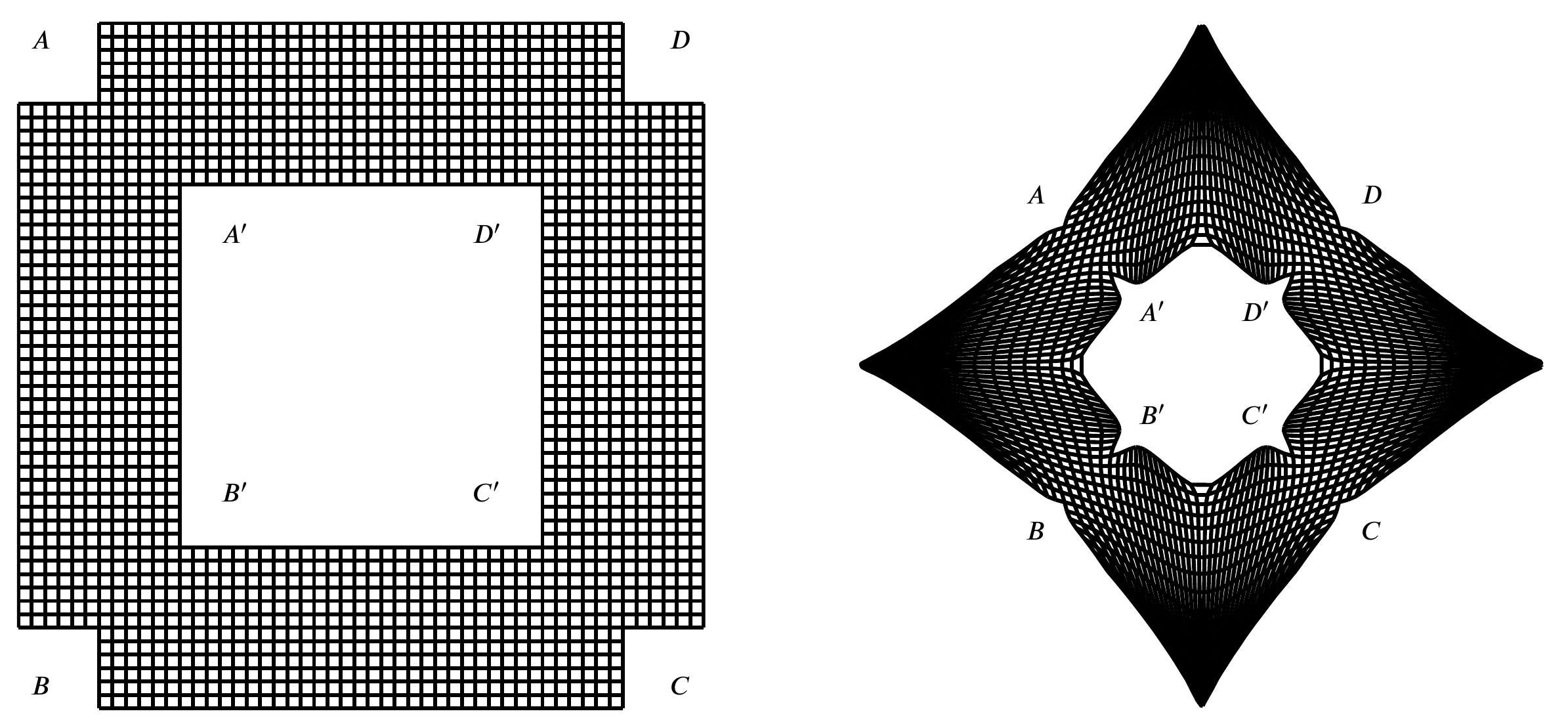}
\caption{Top: Optimal map from the square to ``Pac-Man''.   Bottom: Optimal map between two non-convex shapes.}%
\label{fig:pacman}%
\end{figure}

\begin{figure}[ptb]
\includegraphics[width=.49\linewidth]{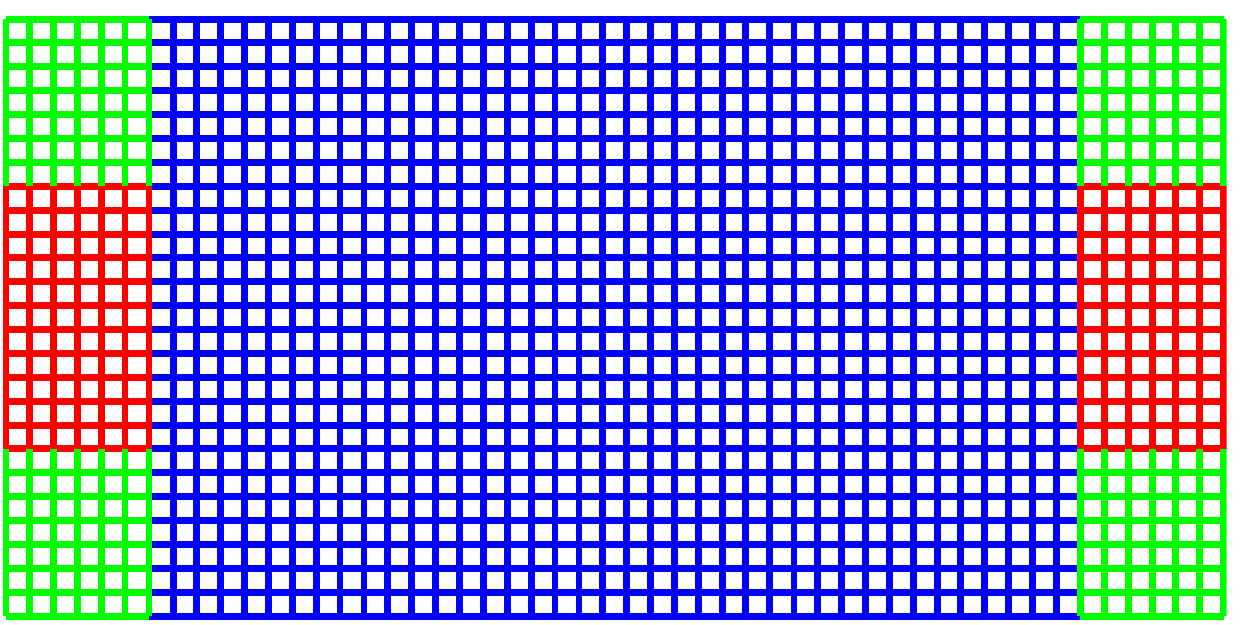}

\includegraphics[width=0.49\linewidth]{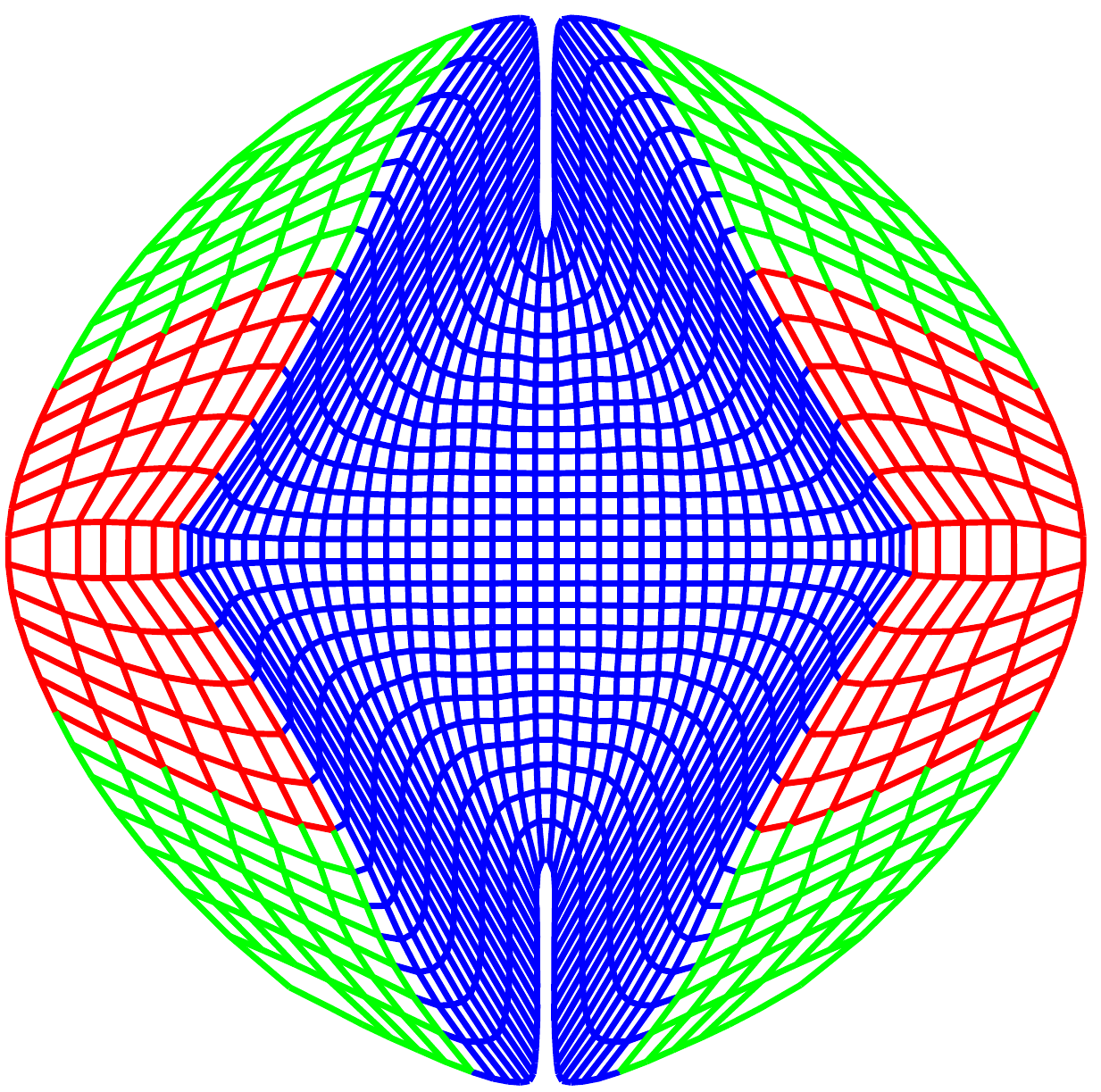}
\includegraphics[width=0.49\linewidth]{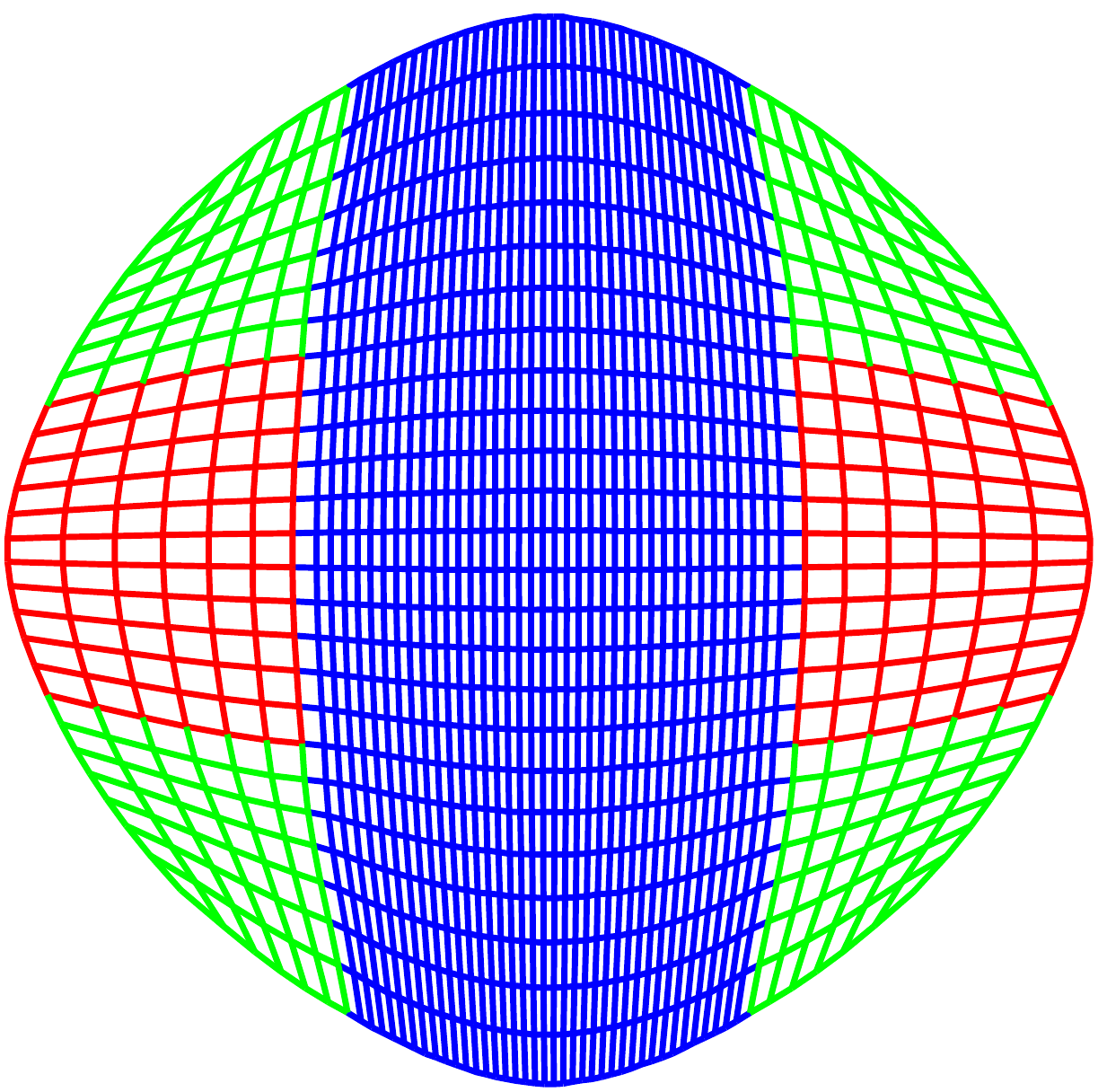}
\includegraphics[width=0.49\linewidth]{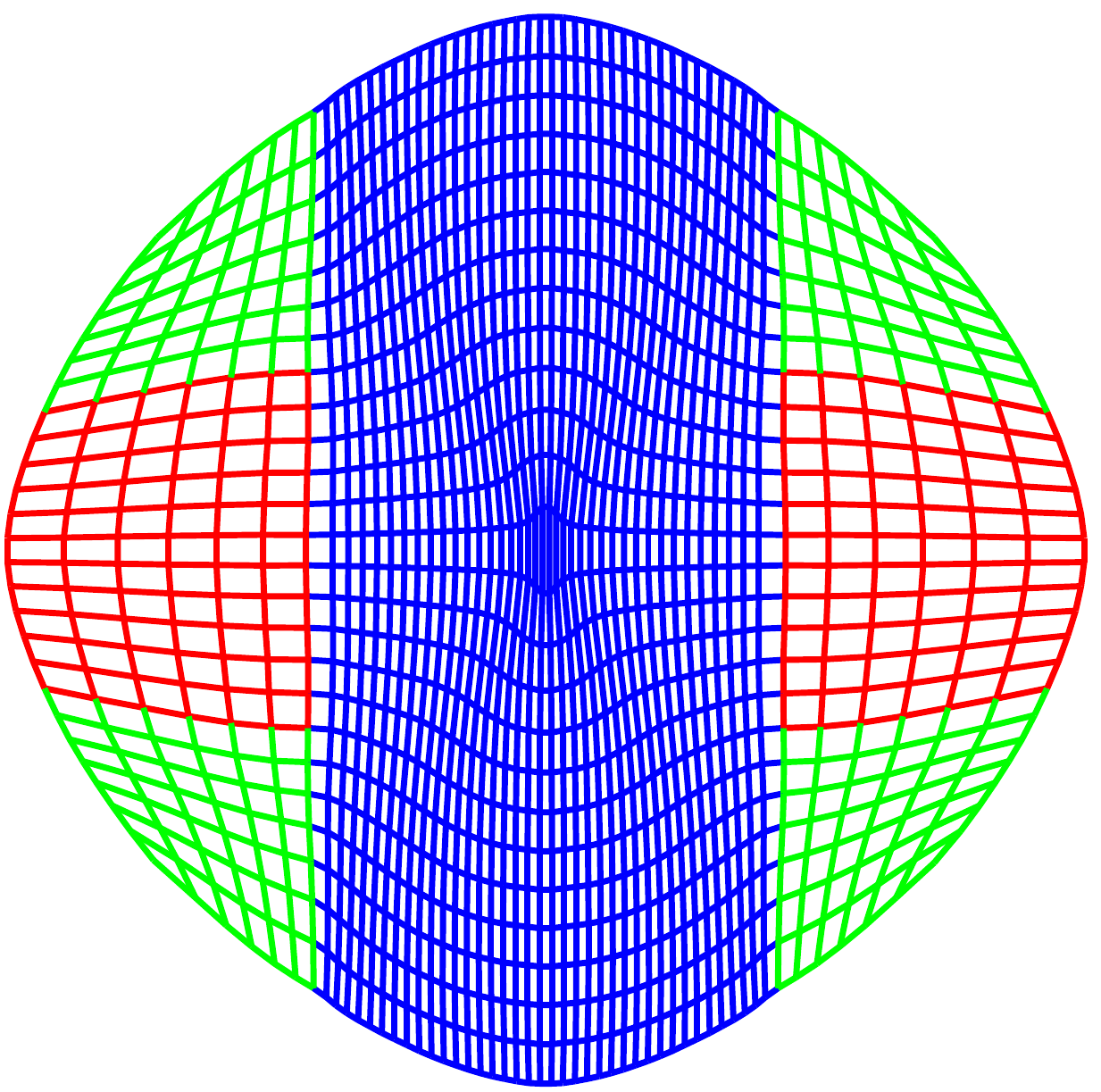}
\includegraphics[width=0.49\linewidth]{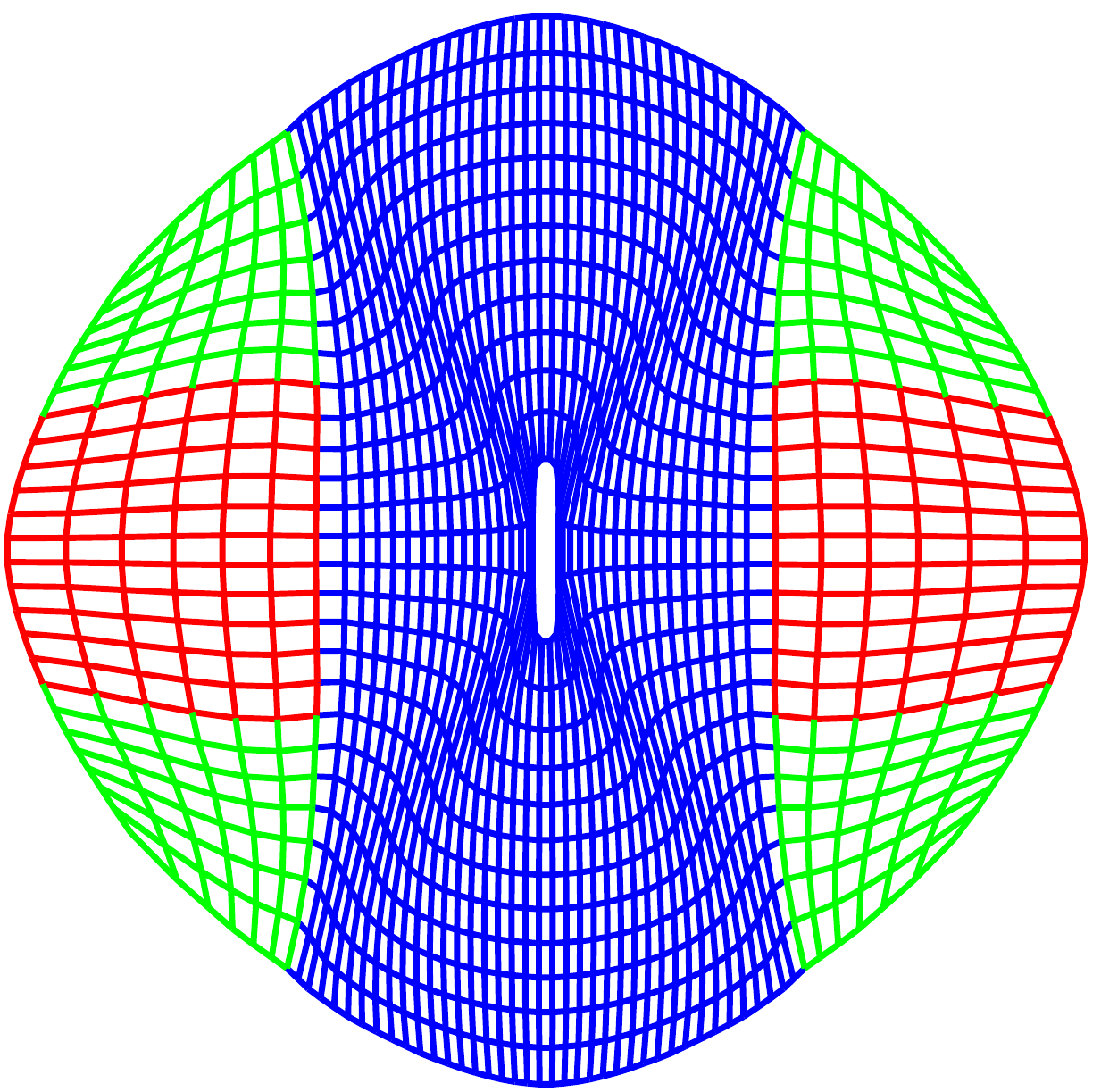}
\caption{Top: source density is piecewise constant on a rectangle, with different values indicated by color.  Optimal map from source to uniform density on the target, with cost 
$c(x,y) = \abs{x-y}^p$, for $p=1.05,2,3,6,$ from left to right and top to bottom.}%
\label{fig:OTVariousCosts}%
\end{figure}

\begin{example}\label{exOT3}
In this example we vary the cost function.  
Refer to Figure~\ref{fig:OTVariousCosts}, which is plotted for a grid of size $256^{2}$.
The source density $\mu$ is  piecewise constant on the rectangle $X = R = [-1,1]\times[-1/2,1/2]$.  It is given, up to normalization,  by
\[
1_{R}
+1_{
\left\{
\abs{x} \geqslant \frac 3 4,
\abs{y} \leqslant \frac 1 2
\right\}
}
+1_{\left\{  \left\vert x\right\vert \geqslant\frac{3}{4},\left\vert
y\right\vert \leqslant\frac{1}{4}\right\}}
\]
The regions are illustrated in the figure by color.   Note that the density is lowest on the blue middle strip $R_m= \{ \abs{x} \leqslant \frac 3 4 \} \subset R$ and highest in the red section.

The target density $\nu$ is uniform on
\[
Y = \left \{ \abs{x}^{8/5} + \abs{y}^{8/5} \le 1 \right \}
\]

The grid lines on $X$ are mapped to  $Y$, along with the colors of the lines.
Since the density is not uniform on $X$, squares coming from higher density areas will be mapped to larger regions, compared to squares from lower density areas.

Consider the second figure, which corresponds to $p=2$.  In this case the low density middle strip of the rectangle, $R_m$,  is mapped to a slightly convex vertical middle strip in the target.  The higher density edges go to corresponding vertical cap shapes in the target.  A square in the center of $R$ is mapped to a rectangle in $Y$ of aspect ratio about 2.

As $p$ increases, a square in the center of $R$ is mapped to a shape in $Y$ with increasing height and decreasing width.  The image of the boundary  vertical lines $\abs{x} = \frac 3 4$ in $R_m$ go from convex ($p=2$) to nearly linear ($p=3$) to concave ($p=6$).

Consider the first image, which corresponds to $p = 1.05$.  In  this case, $R_m$ is mapped to an elongated diamond shape in the middle of $Y$.  The midpoint of the top and bottom edges are mapped into the interior of the target, and their horizontal neighbors are folded in on each other.  The two vertical dents the top and bottom of the target illustrate this folding, which means  the inverse map from $Y$ to $R$ is discontinous.  (Similar results were obtained when the source was changed to  a rounded rectangle.)

\end{example}

\subsection{Partial Optimal Transportation}
The partial optimal transportation problem, studied by Caffarelli and McCann \cite{caffarelli2010free}
and Figalli \cite{figalli2010optimal}, extends the OT problem with quadratic costs to the case where there  is an excess of mass, and only a given fraction is to be transported.  The result is a free boundary problem to determine which mass should be transported  to minimize the total transportation cost.
Uniqueness of the mapping was established in \cite{caffarelli2010free} when $X$ and $Y$ are disjoint, and then by Figalli for when $X\cap Y \not = \emptyset$ when the total mass to be transported is in excess of the measure on the intersection.   Figalli proved local $C^1$ regularity of the free boundaries away from the intersection when $X$ and $Y$ are strictly convex.

The linear programming discretization \eqref{LPfull} is easily modified to include partial mass transportation, by replacing the marginal projection equalities with an inequality, along with a single equality constraint to reflect the total mass to be transported.
\[
\text{Minimize } I[\pi] = \sum_{i=1}^{n}\sum_{j=1}^{m}c_{ij}\pi_{ij},
\]%
\[
\text{Subject \ to:}\left\{
\begin{array}
[c]{cc}%
\text{ }\sum\limits_{j=1}^{m}\pi_{ij}\leqslant \mu_{i},  & i \in \{ 1, \dots, n \},\\
\text{ }\sum\limits_{i=1}^{n}\pi_{ij}\leqslant \nu_{j}, & j\in \{ 1, \dots, m \},\\
\sum_{ij} \pi_{ij} = m.\\
\pi_{ij}\geqslant 0.  \\
\end{array}
\right.
\]
We consider simple geometries for $X$ and $Y$, using uniform measures on the domains. 
We verified that free boundaries are never mapped to free boundaries, which is
consistent with Figalli's prediction.

Each plot shows multiple  free boundaries, corresponding to increasing the total mass to be transported.  See Figure~\ref{fig:FreeBoundary2}.  Note that the apparent kinks in the free boundary lines are plotting artifacts, since we plotted the support squares of the free boundaries without smoothing.  It would also be reasonable to smooth the free boundaries using piecewise linear interpolation.   

\begin{example}
 Here $X$ and $Y$ are upward and downward facing  parabolas, separated by a gap.  The plot shows several the partial optimal transportation free boundaries as the mass to be transported is increased.  The free boundary is both flat, and smooth.  \end{example}

The remaining examples consider $X$ and $Y$ which intersect.   In these cases, 
the total mass to be  transported is  expressed as a multiple of the total mass of the overlap.

\begin{example}
Here $X$ and $Y$ are intersecting squares. The problem was solved on a $512^2$ grid.  For small excess mass ratios, the partial transportation free boundary is small set near the corners of overlap of the squares.  
Only when the excess mass ratio is increased to 1.11 does the free boundary cover the full intersection of the squares. 
\end{example}

\begin{example}
Now $X$ and $Y$ are upward and downward facing parabolas which overlap.  
We see very different results compared to the case with squares.   In this case, 
even for very small excess mass ratios, the partial transportation free boundary contains the intersection  set $X\cap Y$. 
We plotted two figures.  In the first, the mass ratio goes from 1.01 to 1.11, with steps of 0.02.
 For a mass ratio of 1.01, the excess mass is not large compared to the resolution of the grid, but it is clear that as the mass is increased to 1.11, the full intersection is transported.  In the second figure, we steps of 0.04 in the mass ratio, going from 1.05 to 1.25.  At the mass ratio of  1.05, the free boundary is several grid points away from the intersection set. 
\end{example}
 
Notice the difference between the two examples.  At small excess mas ratios,  for the squares, the free boundary is clearly a subset of the intersection, while for the parabolas, it covers the entire intersection.   An obvious difference between the two examples is that the furthest point of the parabolas are not nearly as far away as the furthest points of the squares.    Another way to describe the difference is in terms of the tangents to  the free boundaries.  For the parabolas, the tangents to the free boundaries change very little, from a small positive slope to a small negative one.   For the squares, the tangents to the free boundaries go from nearly horizontal, to nearly vertical.

\begin{figure}[ptb]
\centering\begin{minipage}{1.2\linewidth}
\hspace{-1.25in}
\includegraphics[width=0.6\linewidth]{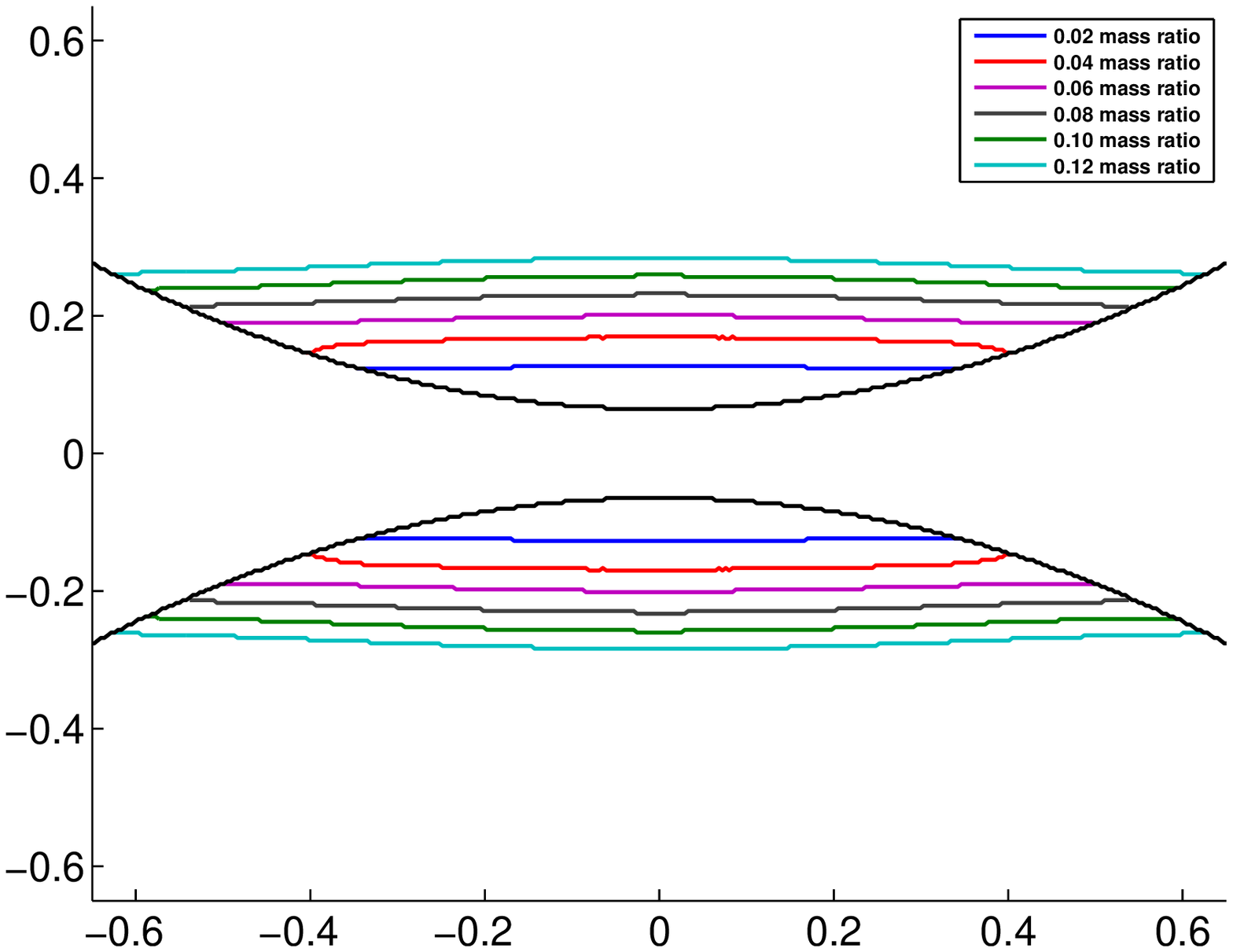}
\includegraphics[width=0.6\linewidth]{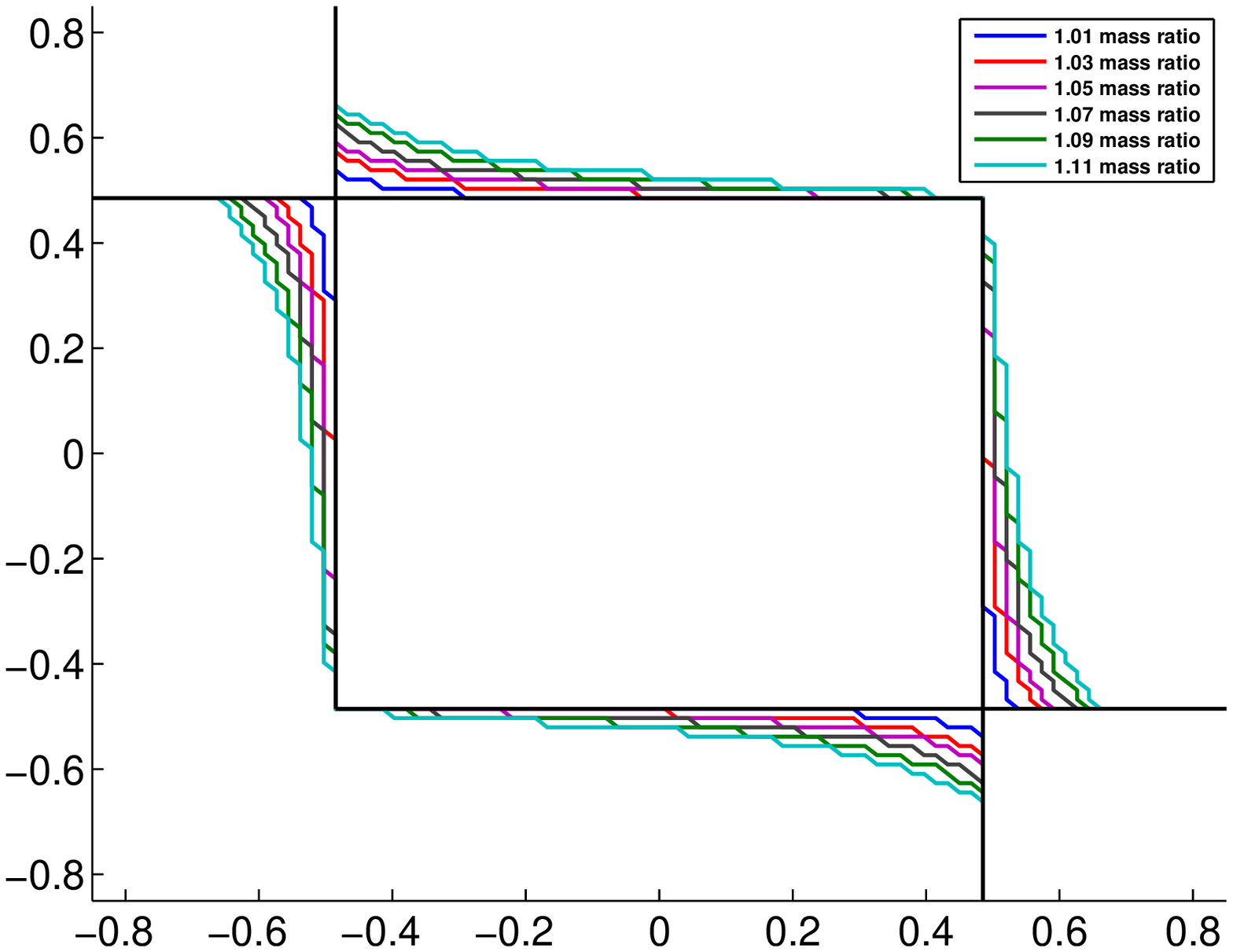}
\end{minipage}
\centering\begin{minipage}{1.2\linewidth}
\hspace{-1.25in}
\includegraphics[width=0.6\linewidth]{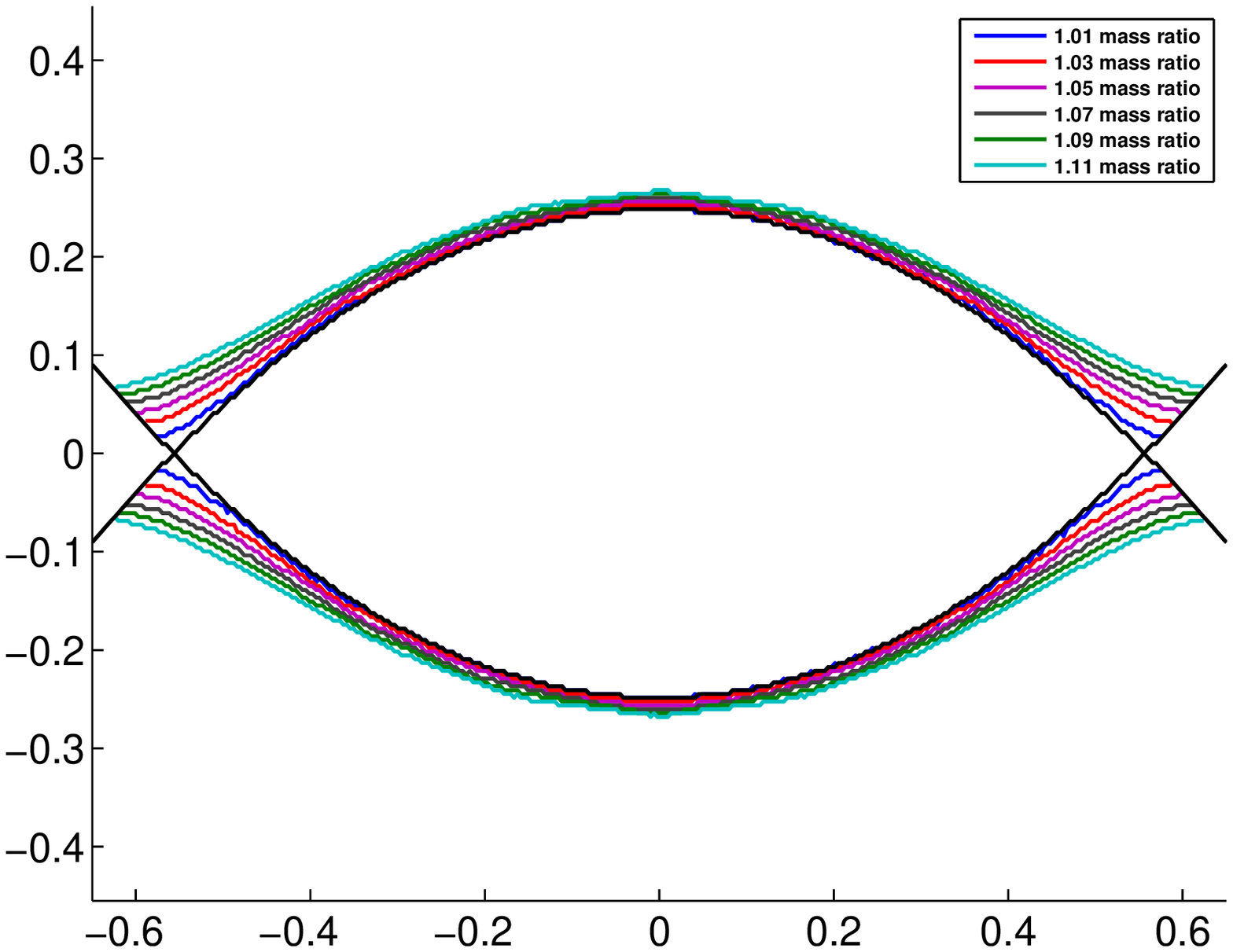}
\includegraphics[width=0.6\linewidth]{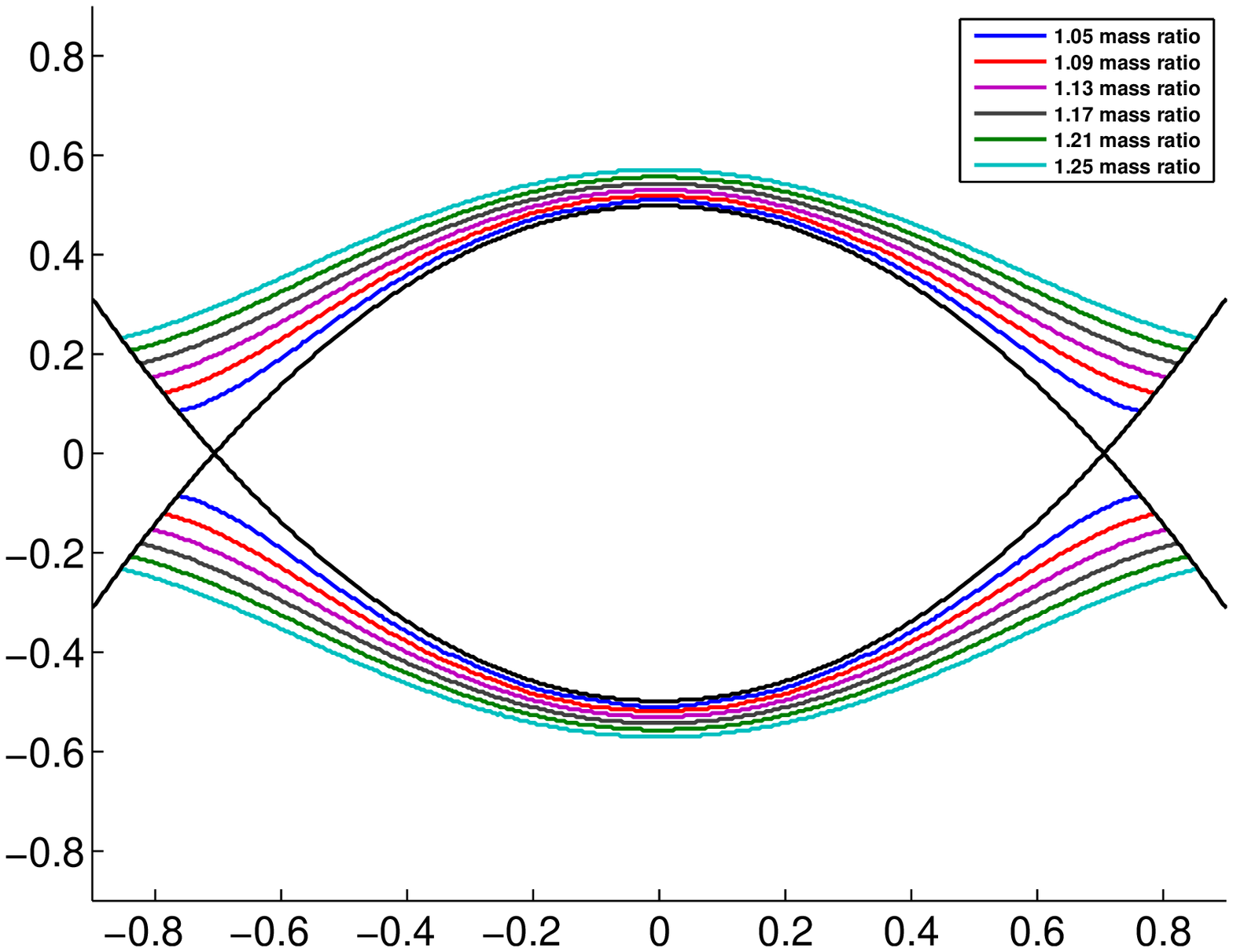}
\end{minipage}
\caption{Partial optimal transportation free boundaries corresponding to different levels of transported mass.  Top left: measures supported on non-overlapping  on parabolas. Top right: Overlapping squares.  Bottom: overlapping paraboloids.
}%
\label{fig:FreeBoundary2}%
\end{figure}

\subsection{Barycenter}\label{sec:barcyenterN}
Given the points $x = (x_1,\dots, x_n)$ and the positive weights $\mu_1, \dots, \mu_n$,  which sum to unity, 
the Euclidean barycenter $\bar \mu =\sum_{i=1}^n \mu_i x_i$ is the minimizer of
\[
I(x) = \sum_{i=1}^n \mu_i \abs{x-x_i}^2.
\]
The Wasserstein distance allows  us to define the barycenter of multiple measures.

By analogy, given the measures $\rho_1, \dots, \rho_n$ and the positive weights $\mu_1, \dots, \mu_n$,  which sum to unity, 
 the barycenter is the minimizer over probability measures of
\[
J(\rho) = \sum_{i=1}^n \mu_i W_2^2(\rho, \rho_i),
\]
where $W_2^2$ is the value of the optimal transportation problem with quadratic costs.  Variants include using other costs,  $c(x,y) = \abs{x-y}^p$.  Refer to
 Agueh and Carlier \cite{agueh2011barycenters}.  The barycenter problem can be discretized as a linear program, see~\cite{carlier2014numerical}.

\begin{example} The barycenter of three shapes is shown in 
 Figure~\ref{fig:CompareCircles}. This  example is taken from \cite[Figure~7]{benamou2015iterative}, where it  was computed on a lower resolution grid.  
  We also computed the barycenters of two rectangles, as in ~\cite{carlier2014numerical} and found very similar results (not plotted).
\end{example}


\begin{example}
Figure~\ref{fig:BC2AnnulusSections} shows the barycenter of three sections of an
annulus.  The support of the barycenter get smaller as $p$ increases.
The barycenter inherits some of the symmetries of the measures.
\end{example}

\begin{figure}[ptb]
\includegraphics[width=0.4\linewidth]{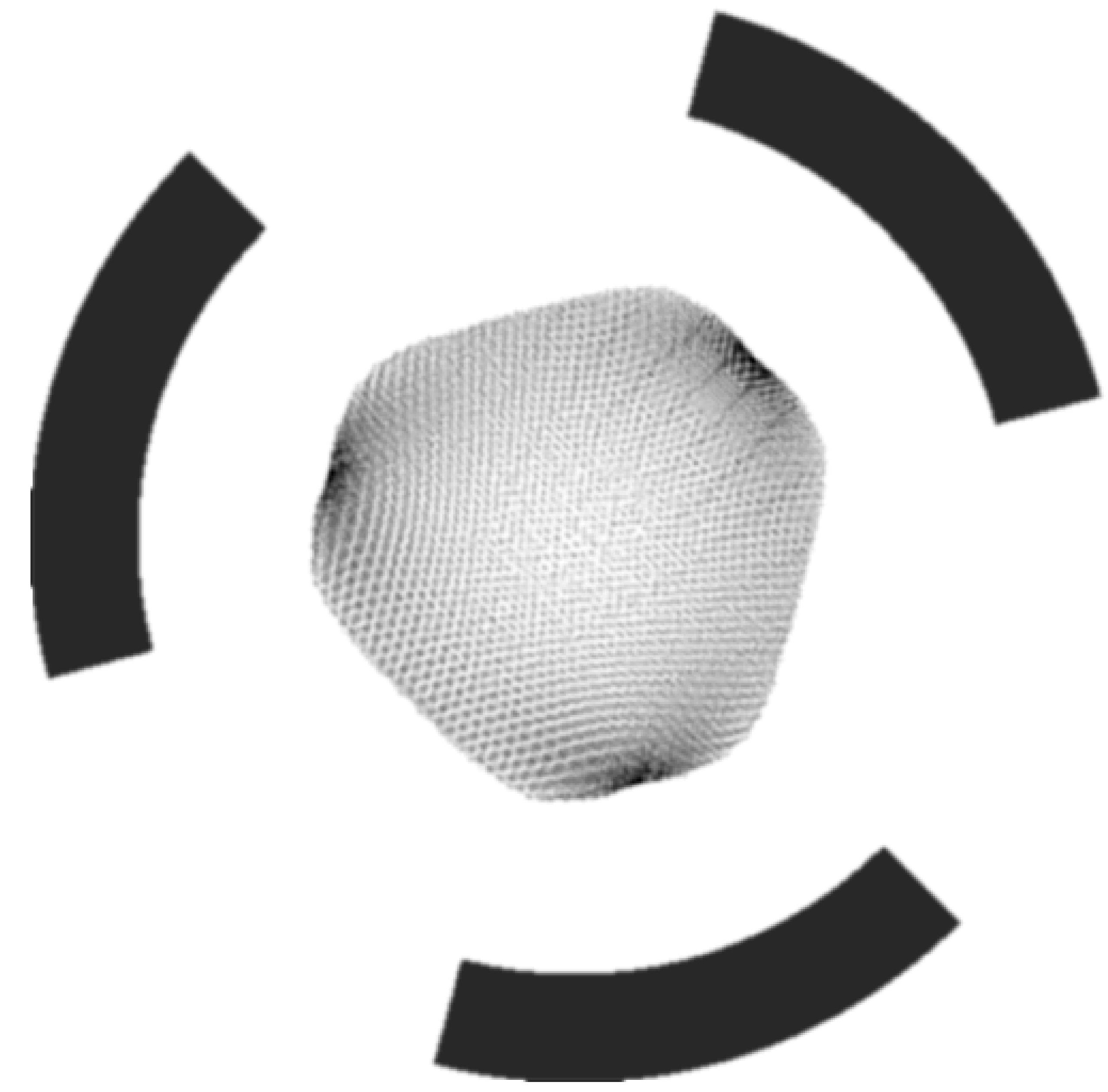} \hspace{0.05\linewidth}\includegraphics[width=0.4\linewidth]{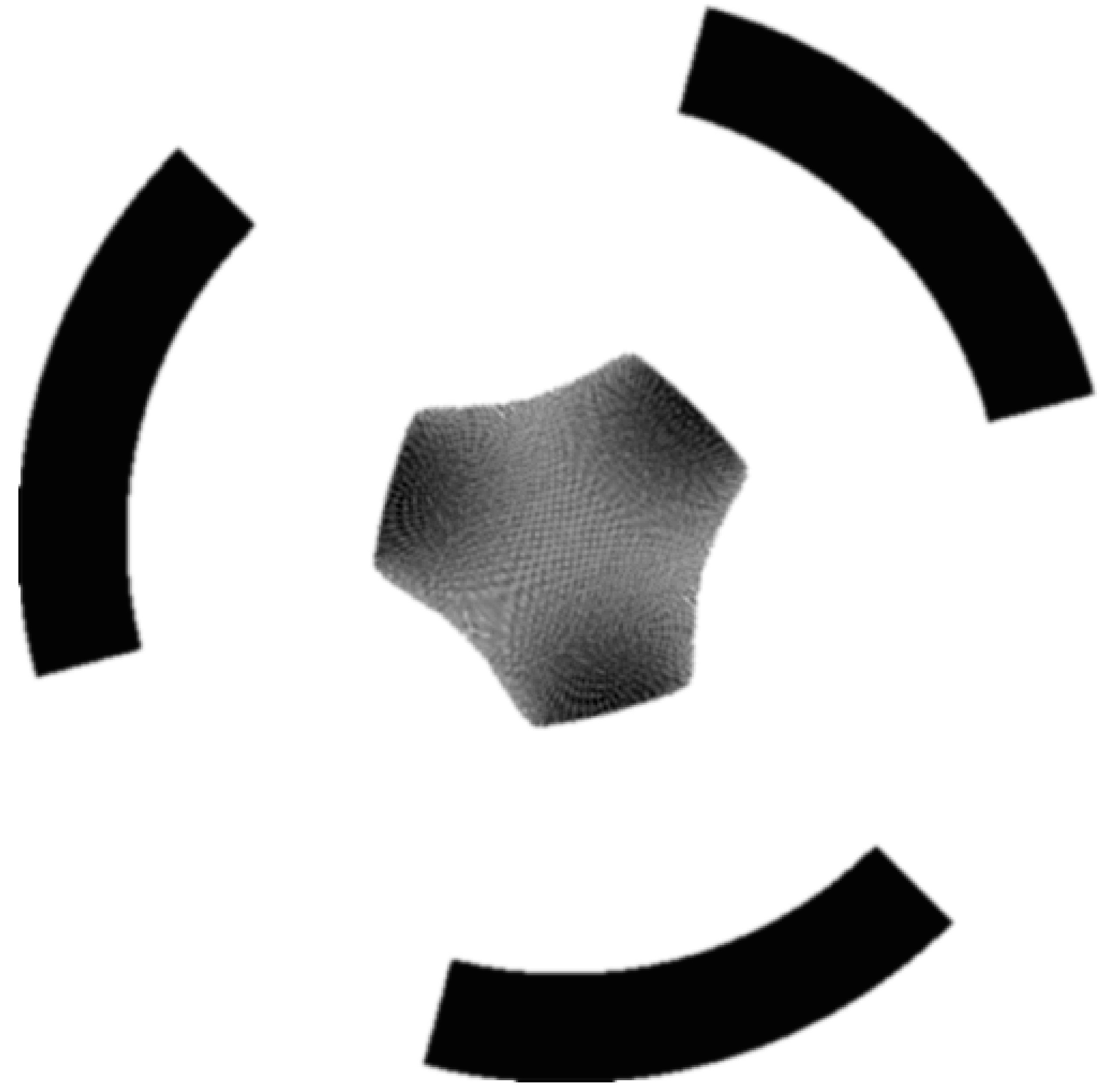}
\includegraphics[width=0.4\linewidth]{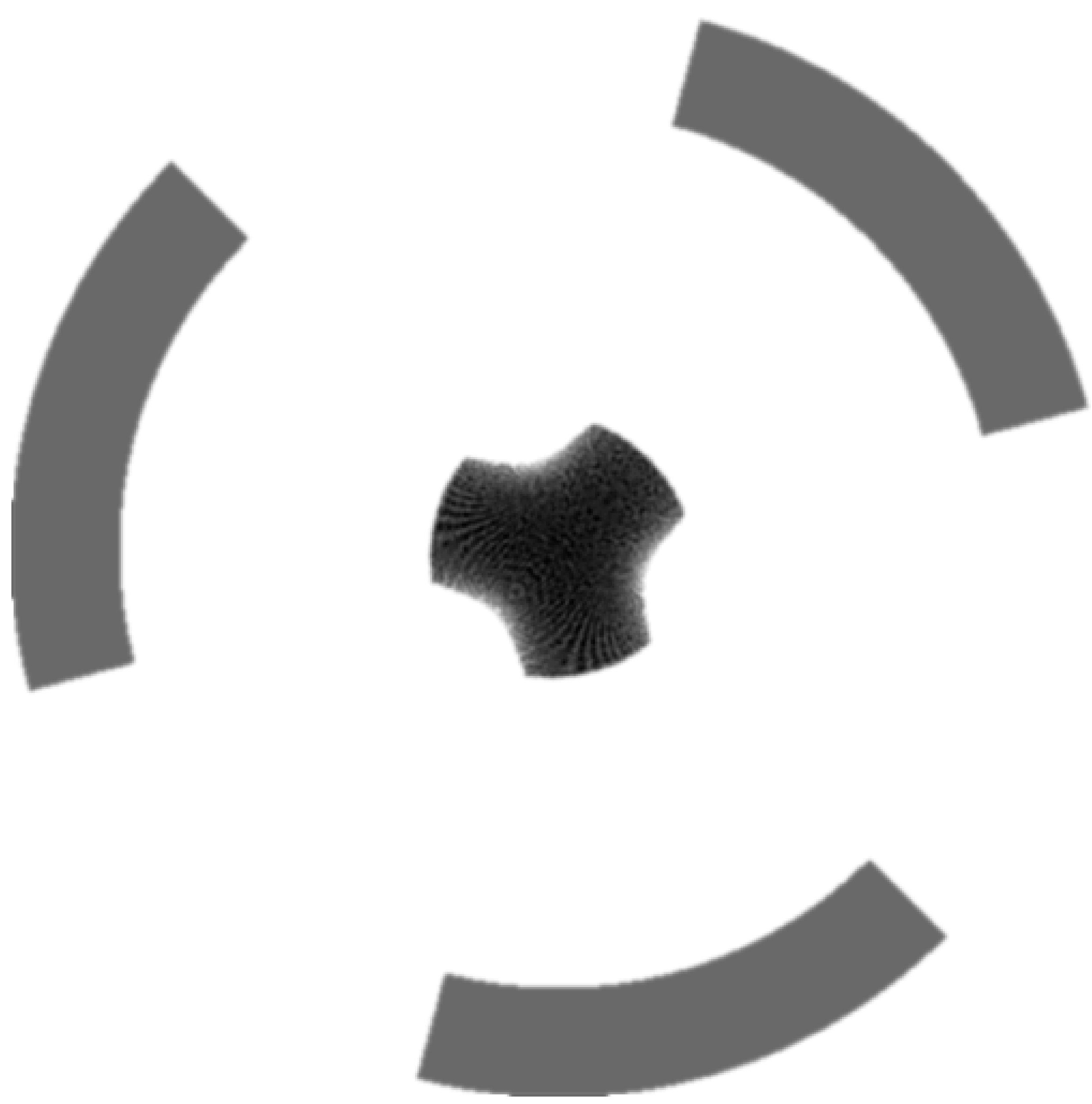} \hspace{0.05\linewidth} \includegraphics[width=0.4\linewidth]{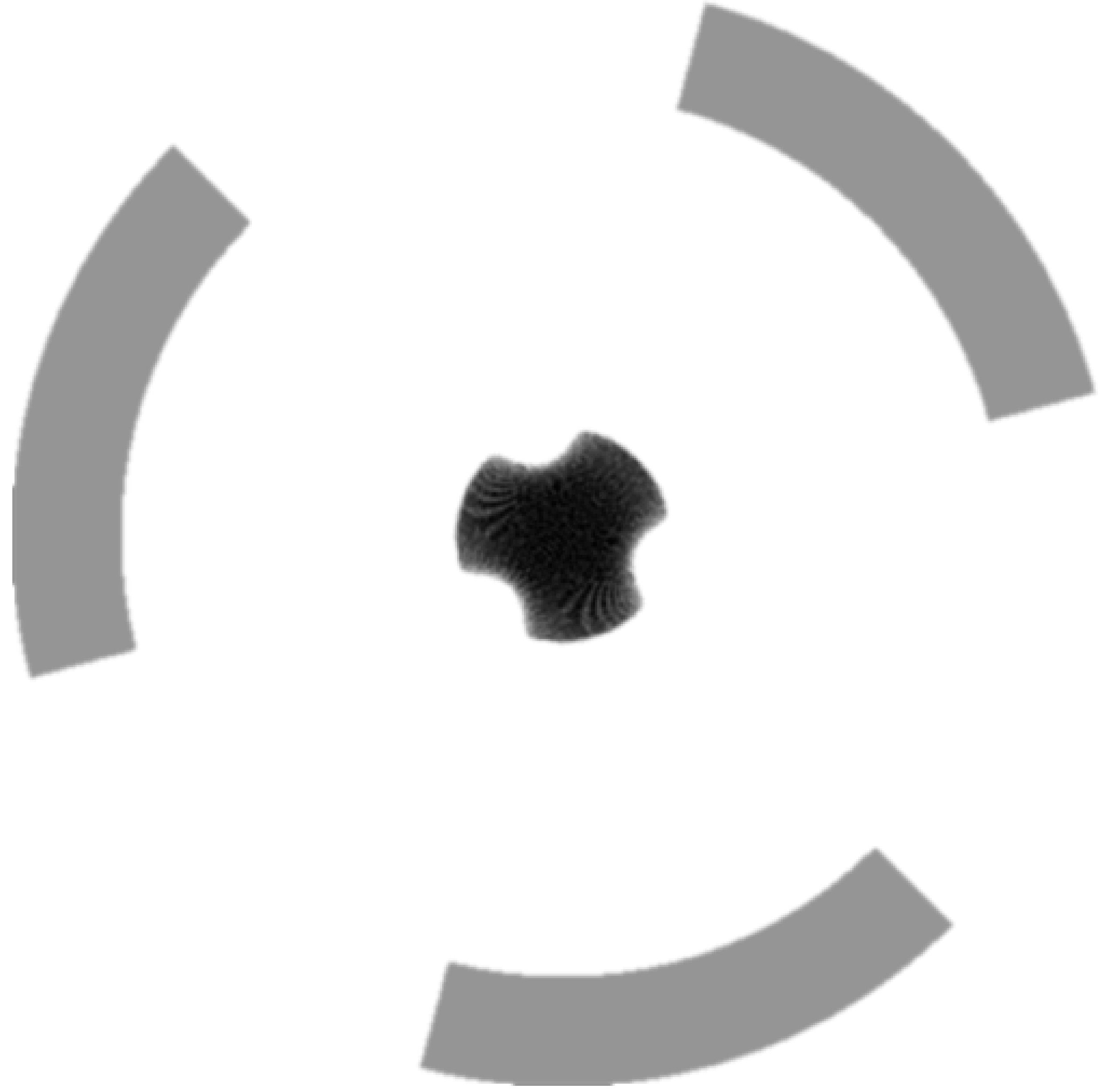}
\caption{Barycenter of three sections of an annulus. From left to
right and top to bottom:  cost $c(x,y) = \abs{x-y}^{p}$, $p=1.05,2,5,9.$ Grid size
$512^{2}.$}%
\label{fig:BC2AnnulusSections}%
\end{figure}

\begin{figure}[ptb]
\includegraphics[width=.9\linewidth]{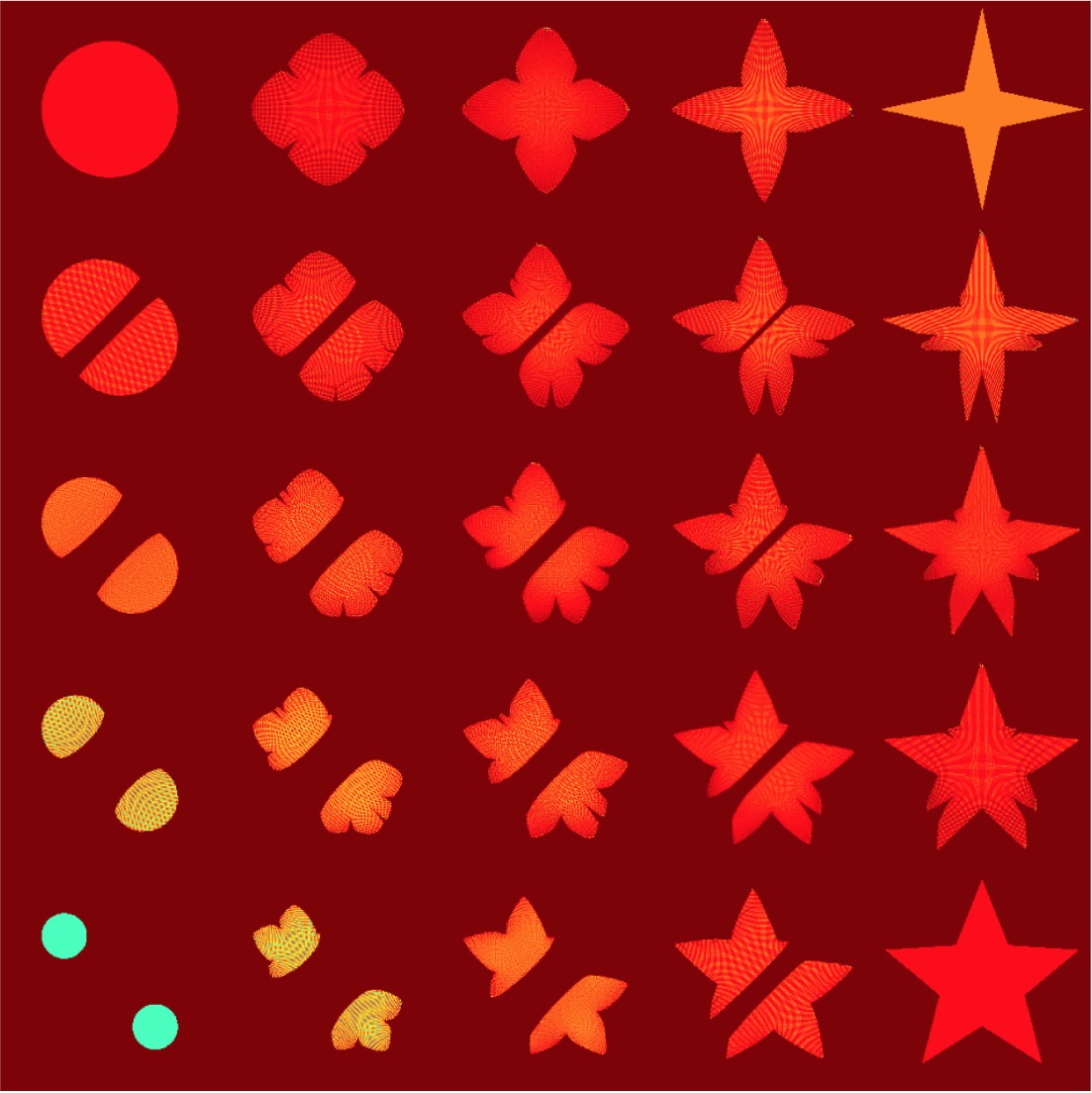}
\caption{Barycenters of the four shapes in the corners, computed on a grid of size $1024^{2}$.  
}%
\label{fig:BC4Shapes}%
\end{figure}

\begin{example}
The next example shows a range of barycenters of uniform measures supported on four shapes, as in \cite{solomon2015convolutional}. Figure~\ref{fig:BC4Shapes}
shows the solution computed on a $1024^{2}$ grid, without smoothing.   
The solution is accurate, and the  boundary of the shapes are sharp, compared to the blurring of boundaries introduced by entropic regularization.
\end{example}

\section{Conclusions}
We introduced an efficient  Linear Programming (LP) method for solving Optimal Transportation  problems on general domains for a wide range of cost functions.    
The method  is linear in complexity (in terms of both time and memory), and allowed for problem sizes of up to half a million variables to be solved in a few minutes on a laptop.  Available parallel LP solvers allow even larger problems to be solved.

The method was accurate enough to  compare solution features for optimal transportation maps between non-convex shapes for different costs functions.   For quadratic costs, the results are comparable to the best available method, but without the restriction of convexity of the target domain.   For other costs, our method is significantly faster and more accurate than available methods. 
We proved weak convergence of the approximation, and implemented the barycentric projection of the measure to improve the accuracy of the mapping.

The method was also applied to generalizations of the Optimal Transportation problem, including barycenters and partial Optimal Transportation.   Other generalizations are possible, as long as these probems can be expressed as Linear Programs.
\bibliographystyle{alpha}
\bibliography{OTbib}

\end{document}